\documentclass[12pt]{article}
\usepackage{amsmath}
\usepackage{amssymb}
\usepackage{url, enumerate}

\newcommand{\eop}{\bigstar}  % end-of-proof

\newcommand{\cf}{{\rm cf}}

\newenvironment{proof}{\noindent{\bf Proof.}}{\par\bigskip}

\newtheorem{THEOREM}{Theorem}[section]

\newtheorem{Conclusion}[THEOREM]{Conclusion}

\newtheorem{Hypothesis}[THEOREM]{Hypothesis}

\newtheorem{LEMMA}[THEOREM]{Lemma}

\newtheorem{Main Theorem}[THEOREM]{Main Theorem}
\newenvironment{main Theorem}{\begin{Main Theorem}} 
{\end{Main Theorem}}

\newtheorem{Theorem}[THEOREM]{Theorem}
\newenvironment{theorem}{\begin{Theorem}}{\end{Theorem}}

\newtheorem{Definition}[THEOREM]{Definition}
\newenvironment{definition}{\begin{Definition}}{\end{Definition}}

\newtheorem{Conventions}[THEOREM]{Conventions}

\newtheorem{Main Definition}[THEOREM]{Main Definition}
\newenvironment{main definition}{\begin{Main Definition}}
{\end{Main Definition}}

\newtheorem{Lemma}[THEOREM]{Lemma}
\newenvironment{lemma}{\begin{Lemma}}{\end{Lemma}}

\newtheorem{Notation}[THEOREM]{Notation}

\newtheorem{Convention}[THEOREM]{Convention}

\newtheorem{Note}[THEOREM]{Note}

\newtheorem{Observation}[THEOREM]{Observation}
\newenvironment{observation}{\begin{Observation}}
{\end{Observation}}

\newtheorem{Remark}[THEOREM]{Remark}
\newenvironment{remark}{\begin{Remark}}{\end{Remark}}

\newtheorem{Question}[THEOREM]{Question}

\newtheorem{Main Fact}[THEOREM]{Main Fact}
\newenvironment{main Fact}{\begin{Main Fact}}{\end{Main Fact}}

\newtheorem{Fact}[THEOREM]{Fact}

\newtheorem{Subfact}[THEOREM]{Subfact}

\newtheorem{Claim}[THEOREM]{Claim}

\newtheorem{Main Claim}[THEOREM]{Main Claim}
\newenvironment{main claim}{\begin{Main Claim}}{\end{Main Claim}}

\newtheorem{Crucial Claim}[THEOREM]{Crucial Claim}
\newenvironment{crucial claim}{\begin{Crucial Claim}}{\end{Crucial Claim}}

\newtheorem{Subclaim}[THEOREM]{Subclaim}

\newtheorem{Sublemma}[THEOREM]{Sublemma}
\newenvironment{sublemma}{\begin{Sublemma}}{\end{Sublemma}}

\newtheorem{Corollary}[THEOREM]{Corollary}
\newenvironment{corollary}{\begin{Corollary}}{\end{Corollary}}

\newtheorem{Example}[THEOREM]{Example}

\newtheorem{Problem}[THEOREM]{Problem}

\newtheorem{Proposition}[THEOREM]{Proposition}

\newtheorem{Conjecture}[THEOREM]{Conjecture}

\newtheorem{Discussion}[THEOREM]{Discussion}

\newenvironment{Proof of the Subfact}
{\noindent{\bf Proof of the Subfact.}}{\par\bigskip}

\newenvironment{Proof of the Theorem}
{\noindent{\bf Proof of the Theorem.}}{\par\bigskip}

\newenvironment{Proof of the Proposition}
{\noindent{\bf Proof of the Proposition.}}{\par\bigskip}

\newenvironment{Proof of the Conclusion}
{\noindent{\bf Proof of the Conclusion.}}{\par\bigskip}

\newenvironment{Proof of the Observation}
{\noindent{\bf Proof of the Observation.}}{\par\bigskip}

\newenvironment{Proof of the Fact}
{\noindent{\bf Proof of the Fact.}}{\par\bigskip}

\newenvironment{Proof of the Lemma}
{\noindent{\bf Proof of the Lemma.}}{\par\bigskip}

\newenvironment{Proof of the Claim}
{\noindent{\bf Proof of the Claim.}}{\par\bigskip}

\newenvironment{Proof of the Corollary}
{\noindent{\bf Proof of the Corollary.}}{\par\bigskip}

\newenvironment{Proof of the Subclaim}
{\noindent{\bf Proof of the Subclaim.}}{\par\medskip}

\newenvironment{Proof of the Main Claim}
{\noindent{\bf Proof of the Main Claim.}}{\par\bigskip}

\newenvironment{Proof of the Crucial Claim}
{\noindent{\bf Proof of the Crucial Claim.}}{\par\bigskip}

\newcommand{\into}{\rightarrow}

\newcommand{\rest}{\upharpoonright}  % restriction
\newcommand{\deq}{\buildrel{\rm def}\over =}
\newcommand{\DD}{{\cal D}}

\newcommand{\FF}{{\cal F}}

\newcommand{\MM}{{\cal M}}
\newcommand{\NN}{{\cal N}}
\newcommand{\PP}{{\cal P}}

\newcommand{\V}{{\bf V}}

\newcommand{\alg}{\mathfrak A}

\newcount\skewfactor
\def\mathunderaccent#1#2 {\let\theaccent#1\skewfactor#2
\mathpalette\putaccentunder}
\def\putaccentunder#1#2{\oalign{$#1#2$\crcr\hidewidth
\vbox to.2ex{\hbox{$#1\skew\skewfactor\theaccent{}$}\vss}\hidewidth}}
\def\name{\mathunderaccent\tilde-3 }

\author{Mirna D\v zamonja\\ School of Mathematics, University of East Anglia\\Norwich, NR4 7TJ, UK
\\\scriptsize{h020@uea.ac.uk} }

\title{Isomorphic universality and the number of pairwise non-isomorphic models in the class of Banach spaces}
\begin{document}
\maketitle
\begin{abstract} We study isomorphic universality of Banach spaces of a given density and a number of pairwise
non-isomorphic models in the same class. We show that in the Cohen model the isomorphic universality number for Banach spaces of density $\aleph_1$ is $\aleph_2$, and analogous results
are true for other cardinals (Theorem \ref{isomorphisms}(1)) and that adding just one Cohen real to any model 
adds a Banach space of density $\aleph_1$ which does not embed into any such space in the ground model (Theorem \ref{Cohenreal}). Moreover, such a Banach space can be chosen to be weakly compactly generated 
(Theorem \ref{Cohenrealc}), giving similar
universality number calculations in the class of wcg spaces in the Cohen model. In another direction, we develop the framework of
{\em natural spaces} to study isomorphic embeddings of Banach spaces and use it to show that a sufficient failure of the generalized continuum hypothesis implies that the universality number of Banach spaces of a given density under
a certain kind of positive embeddings ({\em very positive embeddings}), is high (Theorem \ref{glavni}(1)), and similarly for
the number of pairwise non-isomorphic models (Theorem \ref{glavni}(2)).
\footnote{The author thanks EPSRC for the grants EP/G068720  and EP/I00498 which supported this research and the University of Wroc\l aw in Poland for their
invitation in October 2010 when some of the preliminary results were presented. I especially thank Grzegorz Plebanek for the many productive conversations during the development of this paper.

}
\end{abstract}

\section{Introduction}
We shall be concerned with the isomorphic embeddings of Banach spaces, and in particular with the 
{\em universality number} of this class. For a quasi-ordered class $(\MM, \le)$, the universality number is
defined as the smallest size of $\NN\subseteq \MM$ such that for every $M\in \MM$ there is $N\in \NN$ such that
$M\le N$. In Banach space theory we find many examples of classes whose universality numbers have been
studied, with respect to isomorphic, isometric and other kinds of embeddings. Some specific examples of this are given in sections
 \S\ref{smalldiameter} and \S\ref{vpemb}. Indeed, some of the most classical results in Banach space theory come from these considerations,
 such as the result by Banach and Mazur \cite{Banach}(p.185) that $C([0,1])$ is isometrically universal for all separable Banach spaces. 
The non-separable
case is more complex. We shall note below that the question is only interesting in the context of the failure of the generalized
continuum hypothesis, since GCH automatically gives one universal model for each uncountable density.
The question has received considerable attention over the years, including some recent work such as
that of Brech and Koszmider  \cite{BrKo} who considered Banach spaces of density the continuum and proved that in
the Cohen model for $\aleph_2$ many Cohen reals there is an isomorphically universal Banach space of density 
$\mathfrak c=\aleph_2$. 
In this paper we shall also consider Cohen and Cohen-like models, and we shall prove
results complementary to those of Brech and Koszmider. Specifically, in the model they considered, we shall see that 
the isomorphic universality number for Banach spaces of density $\aleph_1$ is $\aleph_2$, and we shall also show that analogous results
are true for other cardinals (Theorem \ref{isomorphisms}(1)). We shall also see that adding just one Cohen real to any model adds a Banach space of density $\aleph_1$ which does not embed into any such space in the ground model  (Theorem \ref{Cohenreal}).
We note that in \cite{BrKo2} it is stated (pg. 1268), without proof, that Koszmider and Thompson noted that a version of the proof from 
\cite{BrKo} gives a model where there is no isomorphically universal Banach space of density $\aleph_1$. A further relation
with the work of Brech and Koszmider is that in \cite{BrKo2} they gave a model where $2^{\aleph_0}=\aleph_2$ and  the universality number
for the class of weakly compactly generated Banach spaces of density $\aleph_1$ is $\aleph_2$, while we show that the
same is true in the classical Cohen model for $\aleph_2$ Cohen reals (Thereom \ref{wcgfinal}).

A question related to that of universality is that of the number of pairwise non-isomorphic models. This is a well studied question in
model theory (see Shelah's \cite{Sh-c}) and it has received considerable attention in Banach space theory of separable
Banach spaces, see Rosendal \cite{Rosendal} for an excellent survey of this and related problems. In the non-separable
context, there do not seem to be many results available. As a consequence of our non-universality results we show that in the Cohen model described above
there is a family of $2^{\aleph_1}$ many Banach spaces of density $\aleph_1$ which are pairwise non-isomorphic,
and the same is true for other cardinals (Theorem \ref{isomorphisms}(2)) and wcg spaces (Thereom \ref{wcgfinal}).

Moving away from forcing results, which give results valid in specific models of set theory, we would ideally like to 
establish results in the form of implications from cardinal arithmetic. An example is the above mentioned result
that under GCH there is an isomorphically universal Banach spaces for every density. In the absence of GCH results
of this type may be obtained by using the pcf theory. We explore this direction in \S\ref{natural} -\S\ref{proofCL}, where we prove
that for a special kind of embedding ``very positive'' (see Definition \ref{strictlypositive}) a sufficient failure of GCH implies
the nonexistence of universal Banach spaces in the class of spaces $C(K)$ for $K$ 0-dimensional,
(Theorem \ref{glavni} (1)) as well as the existence of a large number of pairwise non-isomorphic models (Theorem \ref{glavni} (2)).
While the use of forcing methods in the isomorphic Banach space theory is not new, as explained above, to our knowledge no pcf results
have previously been obtained in this context. Our result required a development of a new framework,
which we call {\em natural spaces}. These are model-theoretic structures which we use to talk about the
spaces of the form $C(K)$ indirectly. Model theory connected to pcf theory was used by Shelah and Usvyatsov
 \cite{ShUs} to study the isometric theory of Banach spaces and prove negative universality results about them.  We review
 these results below. They can also serve as a motivation to a question that has interested us throughout, which is to
 what extent the universality number of the class of Banach spaces of given density depends on the kind of embedding 
 considered. Let us now pass to some background results and notation. Throughout, $\kappa$ stands for an infinite cardinal.

By combining the Stone duality theorem, the fact that any Banach space $X$ is isometric to a subspace of $C(B_{X^\ast})$
and that $B_{X^\ast}$ has a totally disconnected continuous preimage, Brech and Koszmider 
proved the following:

\begin{Fact}\label{classicalf} [Fact 1.1, \cite{BrKo}] (1) The universality number of the class of Boolean algebras of size $\kappa$ is greater or equal
to the universality number of the class of Banach spaces of density $\kappa$ with
isometric embeddings, which is  greater or equal than
the universality number of the class of Banach spaces of density $\kappa$ with
isomorphic embeddings.

{\noindent (2)} The class of spaces of the form $C({\rm St}(\alg))$ for $\alg$ a Boolean algebra of size $\kappa$ is 
isometrically universal for the class of Banach spaces of density $\kappa$, and in particular its universality number with
either isometric or isomorphic embeddings is the same as the universality number of the 
whole class of Banach spaces of density $\kappa$.
\end{Fact}

Fact \ref{classicalf} is only interesting in the context of uncountable $\kappa$, since for $\kappa=\aleph_0$ we have
a universal Boolean algebra as well as an isometrically  universal Banach space, as explained above.
On the other hand, it is known from the classical model theory (see \cite{ChKe} for saturated and special models) that in the presence of GCH there is a universal Boolean algebra at every
uncountable cardinal, so the questions of universality for the above classes are interesting in the context of the failure of
the relevant instances of GCH. Negative universality results for Boolean algebras are known to hold
when GCH fails sufficiently by the work of Kojman and Shelah \cite{KjSh409}, and in Cohen-like extensions by the work of
Shelah (see \cite{KjSh409} for a proof). Shelah and Usvyatsov proved in \cite{ShUs} that in the models where the negative universality results that were obtained for Boolean algebras in \cite{KjSh409} hold, the same negative universality results hold for Banach spaces under isometric embeddings. The smallest cardinal at which these results can apply is $\aleph_2$.
For example, if $\lambda$ is a regular cardinal greater than $\aleph_1$ but smaller
than $2^{\aleph_0}$ (so $2^{\aleph_0}\ge \aleph_3$) there is no universal under isometries Banach space of density $\lambda$.

On the basis of what is known in the literature and what we obtain here, it is interesting to note that no known result differentiates
between the universality number of Banach spaces of a given density under isometries or under isomorphisms. Furthermore,
it is not known how to differentiate them from the universality number of Boolean algebras.

\begin{Conjecture}\label{conj1} The universality number of the class of Banach spaces of density $\kappa$ with
isomorphic embeddings is the same
as the universality number of the class of Boolean algebras of size $\kappa$.
\end{Conjecture}

It follows from the above discussion that Conjecture \ref{conj1} would improve Fact \ref{classicalf} (1)
and it would imply the negative universality results of Shelah and Usvyatsov. For all we know at this point, 
Conjecture \ref{conj1} could be a theorem of ZFC, that is, it is not known to fail at any $\kappa$ even consistently.
A particular case of Conjecture \ref{conj1}
is the following Conjecture \ref{conj2}, which summarizes the most interesting case from the point of view of Banach
space theory:
 
\begin{Conjecture}\label{conj2} The universality number of the class of Banach spaces of density $\kappa$ with
isomorphic embeddings is the same as the universality number of the class of Banach spaces of density $\kappa$ with
isometric embeddings.
\end{Conjecture}

In Banach space theory one studies other kinds of embeddings but isomorphisms and isometries, an example
is described in \S\ref{smalldiameter}. It would be interesting to refine the above conjectures in other contexts, for example in the
context of embeddings with a fixed Banach-Mazur diameter, see \S\ref{smalldiameter} for a discussion.
We now finish the introduction by giving some background information for the readers less familiar with Banach space theory.

\begin{definition} A {\em Banach space} is a normed vector space complete in the metric induced by the norm.
A linear embedding $T:X\to Y$ between Banach spaces is {\em an isometry} if for every $x\in X$ we have
$||x||=||Tx||$, where we use $Tx$ to denote $T(x)$. A linear embedding $T:X\to Y$ between Banach spaces is {\em an isomorphism} if there is a constant $D>0$ such that for every $x\in X$ we have
$\dfrac{1}{D} ||x||\le ||Tx||\le D||x||$. 
\end{definition}

\begin{remark} Every isometry is an isomorphism. An isomorphism is in particular an injective continuous function and in fact, 
a linear map $T$ is an isomorphism iff  both $T$ and $T^{-1}$ are linear and continuous.

For $T$ an isomorphism we define $||T||\deq \sup \{||T(f)||:\,||f ||=1\}$ and we say that the {\em Banach-Mazur diameter} of $T$ is
the quantity $||T||\cdot ||T^{-1}||$.
\end{remark}

Throughout the paper letters $\frak A$ and $\frak B$ will be used for Boolean algebras, $\kappa,\lambda$ for
infinite cardinals and $K$ and $L$ for
compact spaces. We shall write ${\rm St}(\alg)$ for the Stone space of a Boolean algebra $\alg$. Let us note that Fact \ref{classicalf}
implies

\begin{observation}\label{gap} The universality number of Banach spaces of density $\kappa$, under any kind of embeddings, is either 1 or
$\ge\kappa^+$.
\end{observation}

This is so because if for any $\alpha^\ast\in [1,\kappa^+)$ we had that $\{X_\alpha:\,\alpha<\alpha^\ast\}$ were a universal family
of Banach spaces of density $\kappa$,
then we could assume that each $X_\alpha=C({\rm St}(\alg_\alpha))$ for some Boolean algebras $\alg_\alpha$ of size
$\kappa$. Therefore we could find a single algebra $\alg$ of size
$\kappa$ such that all $\alg_\alpha$ embed into it \footnote{simply by freely generating an algebra by a disjoint union of all $\alg_\alpha$}
and hence then $C({\rm St}(\alg))$ would be a single universal Banach space of density $\kappa$.

\section{Universals in Cohen-like extensions}\label{Cohenlike}
It is a well known theorem of Shelah (see \cite{KjSh409}, Appendix, for a proof) that in the extension obtained by adding a regular 
$\kappa\ge \lambda^{++}$ number of
Cohen subsets to a regular $\lambda$ over a model of GCH, the universality number for models of size $\lambda^+$ for
any unstable complete first order theory is $\kappa=2^\lambda$, so the maximal possible. This in particular applies to Boolean algebras. We
explore to what extent this can be adapted to 
the class of Banach spaces of density $\lambda$ by looking at variants of Cohen forcing. Our results show that Conjecture \ref{conj1}
holds in this special case, that is, in these circumstances we obtain the strongest possible negative universality result.

Let $\lambda=\lambda^{<\lambda}$.
Fix a set $A=\{a_i:\,i<\lambda^+\}$ of indices which we shall assume forms an unbounded co-unbounded subset of
$\lambda^+$.

\begin{definition} The forcing  ${\mathbb P}(\lambda)$ consists of Boolean algebras $p$ generated on a subset of $\lambda^+$ by some subset $w_p$ of $A$ of size $<\lambda$
satisfying $p\cap A=w_p$.
The ordering on  ${\mathbb P}$ is given by $p\le q$ if $p$ is embeddable as a subalgebra of $q$ and the embedding fixes
$w_p$, where in our notation $q$ is the stronger condition.
\end{definition}

Let us check some basic properties of the forcing ${\mathbb P}(\lambda)$, reminding the reader of the following notions:

\begin{definition}\label{ccclosed} (1) A forcing notion ${\mathbb P}$ is said to be $\lambda^+$- {\em stationary cc} if
for every set $\{p_i:\,i<\lambda^+\}$ of conditions in ${\mathbb P}$, there is a club $C$ and a regressive function $f$
on $C$ satisfying that for every $i,j<\lambda^+$ of cofinality $\lambda$ satisfying $f(i)=f(j)$, the conditions $p_i$ and $p_j$ are
compatible.

{\noindent (2)} A subset $Q$ of a partial order $P$ is {\em directed} if every two elements of $Q$ have an upper bound in
$P$.  A forcing notion ${\mathbb P}$ is $(<\lambda)$-directed closed if every directed $Q\subseteq {\mathbb P}$
of size $<\lambda$, has an upper bound in ${\mathbb P}$.
\end{definition}

\begin{lemma}\label{properties}
(1)  ${\mathbb P}(\lambda)$ satisfies the $\lambda^+$-stationary cc and is $(<\lambda)$-directed closed.

{\noindent (2)} ${\mathbb P}(\lambda)$ adds a Boolean algebra of size $\lambda^+$ generated by $\{a_i:\,i<\lambda^+\}$.

\end{lemma}

\begin{proof} For part (1), suppose that $\{p_i:\,i<\lambda^+\}$ are conditions in ${\mathbb P}(\lambda)$. Let us denote
by $w_i$ the set $w_{p_i}$ and let us consider it in its increasing enumeration. Let the {\em isomorphism type} of $w_i$
be
determined by the order type of $w_i$ and the Boolean algebra equations satisfied between the elements of $w_i$, denote this by $t(w_i)$. Note that by 
$\lambda^{<\lambda}=\lambda$ the cardinality of the set $\mathcal I$
of isomorphism types is exactly $\lambda$, so let us fix a bijection 
$g:\,{\mathcal I}\times [\lambda^+]^{<\lambda} \to \lambda^+$. 
Note that there is a club $C$ of $\lambda^+\setminus 1$ such that for every point $\gamma$ in $C$ of cofinality $\lambda$ we have
$i<\gamma$ iff $w_i\subseteq \gamma$ and $g``(\,{\mathcal I}\times\gamma)\subseteq\gamma$. Define $f$ on $C$ by letting $f(j)=g(t(w_j), w_j\cap j)$ for $j$ of cofinality $\lambda$ and $0$ otherwise. Hence $f$ is regressive on $C$
and if $i<j$ both in $C$ have cofinality $\lambda$ and satisfy $f(i)=f(j)$ then we have that $w_i\cap i=w_j\cap j= w_i\cap w_j$ and that $p_i$ and $p_j$
satisfy the same equations on this intersection. Hence $p_i$ and $p_j$ are compatible.

For the closure, note that a family of conditions in ${\mathbb P}(\lambda)$ being $(<\lambda)$-directed in particular implies
that the conditions in any finite subfamily agree on the equations involving any common elements. Therefore the union of the
family generates a Boolean algebra which is a common upper bound for the entire family.  

For part (2), note that if $p\in {\mathbb P}$ and $a_i$ is not in $w_p$, then we can extend $p$ to 
$q$ which is freely generated by $w_p \cup\{a_i\}$, except for the equations present in $p$. Hence the set
$\DD_i=\{p:\,a_i\in w_p\}$ is dense. Note that if $p,q\in G$ satisfy
$p\cap A=q\cap A$, then $p$ and $q$ are isomorphic over $w_p$. Let $H$ be a subset of $G$ which is obtained by
taking one representative of each isomorphism class with the induced ordering, hence $H$ is still a filter in ${\mathbb P}$ and it 
intersects every $\DD_i$. By the downward closure of $G$ it follows that for
every $F$ a subset of $A$ of size $<\lambda$, there is exactly one element $p_F$ of $H$ with $w_{p_F}=F$. (Note that no 
new bounded subsets of $\lambda$ are added by ${\mathbb P}$, since (1) holds).
Now note that $\{p_F:\,F\in [A]^{<\lambda}\}$ with the relation of embeddability over $A$ form an inverse system of Boolean
algebras directed by $([A]^{<\lambda}, \subseteq)$. Let $B$ be the inverse limit of this system (see \cite{Ha} for an explicit
construction of this object), so $B$ is generated by $\{a_\alpha:\,\alpha<\lambda\}$ and it embeds every element of $H$.
(We shall call this algebra ``the" generic algebra. It is unique up to isomorphism).
$\eop_{\ref{properties}}$
\end{proof}

Note that in the case of $\lambda=\aleph_0$, we in particular have that the forcing is ccc.

\begin{theorem}\label{nonembedding} If $\alg$ is the generic Boolean algebra for some  ${\mathbb P}(\lambda)$, then there is no Banach space $X_\ast$ in the ground model such 
that $C({\rm St}(\alg))$ isomorphically embeds into $X_\ast$ in the extension by ${\mathbb P}(\lambda)$.
\end{theorem}

\begin{proof} Let us fix a $\lambda$.
Suppose for a contradiction that there is an $X_\ast$ and embedding $T$ in the extension contradicting the statement of
the theorem. Let $p^\ast$
force that $\name{T}$ is an isomorphic embedding and, without loss of generality $p^\ast$ decides a positive constant
$c$ such that $1/c \cdot ||x||\le ||Tx||\le c\cdot  ||x||$ for every $x$. Let $n_\ast$ be a positive integer such that $c< n_\ast$.
For each $i<\lambda^+$ let $p_i\ge p^\ast$ be such that 
$a_i\in w_{p_i}$, and without loss of generality $p_i$ decides the value $x_i=T (\chi_{[a_i]})$. Now let us perform
a similar argument as in the proof of the chain condition, and in particular by passing to a subfamily of the
same cardinality we can assume that
the sets $w_{p_i}$ form a $\Delta$-system with a root $w^\ast$  . Since $a_i\in w_{p_i}$, we can assume
that $w_i\neq w^\ast$, and in particular that $a_i\in w_i\setminus w^\ast$. Let us take distinct
$i_0,  i_1, \dots   i_{n^2_\ast}$. Since there
are no equations in $p_i$ or $p_j$ that connect $a_i$ and $a_j$ for $i\neq j$, we can find two conditions $q'$ and $q''$
which both extend $p_{i_0},\ldots p_{n^2_\ast}$ and such that in $q'$ we have $a_{i_0} \le \ldots \le a_{i_{n^2_\ast}}$
and in $q'$ we have that $a_{i_0},\ldots a_{i_{n^2_\ast}}$ are pairwise disjoint.
Hence $q'$ forces that $\chi_{[a_{i_0]}}+\ldots \chi_{[a_{i_{n^2_\ast}}]}$ has norm $n^2_\ast+1$ and
 $q''$ forces that it has norm 1. Therefore the vector $x_{i_0}+\ldots x_{i_{n^2_\ast}}$ must have norm both $>n_\ast$ and 
 $<n_\ast$ in $X_\ast$, which is impossible.
$\eop_{\ref{nonembedding}}$
\end{proof}

The ideas in the proof of Theorem \ref{nonembedding} which use the possibility to extend Cohen-like conditions in incomparable ways are present in the proof by Brech and Koszmider  (\cite{BrKo}, Theorem 3.2) that $l^\infty/c_0$ is not a universal Banach space
for density $\mathfrak c$ in the extension by $\aleph_2$ Cohen reals, and the earlier proof by D\v zamonja and Shelah 
(\cite{DjSh614}, Theorem 3.4) that in the model considered there
there is no universal normed vector space over ${\mathbb Q}$ of size 
$\aleph_1$ under isomorphic vector space embeddings. In fact in the case of $\lambda=\aleph_0$ we can expand on these ideas
to prove the following Theorem \ref{Cohenreal}. We shall use a simplified $(\omega,1)$-morass constructed in ZFC by Velleman
in \cite{Vellemanmorass}\footnote{In a related work, Velleman in \cite{Velleman} showed that an effect of a Cohen real on such a morass is to add a Souslin tree.}.
 
\begin{theorem}\label{Cohenreal} 
Forcing with one Cohen real adds a Boolean algebra $\alg$ of size $\aleph_1$ such that
in the extension the Banach space $C({\rm St}(\alg))$ does not isomorphically embed into any space $X_\ast$ in
the ground model of density $\aleph_1$.
\end{theorem}

\begin{proof}  Let us recall
that a {\em simplified $(\omega,1)$-morass} is a system $\langle \theta_\alpha:\,\alpha\le\omega\rangle, \langle \FF_{\alpha,\beta}:\,\alpha<\beta\le\omega\rangle$ such that
\begin{enumerate}
\item for $\alpha<\omega$, $\theta_\alpha$ is a finite number $>0$, and $\theta_{\omega}=\omega_1$,
\item for $\alpha<\beta<\omega$, $\FF_{\alpha,\beta}$ is a
finite set of order preserving functions from $\theta_\alpha$ to $\theta_\beta$,
\item $\FF_{\alpha,\omega}$ is a set of order preserving functions from $\theta_\alpha$ to $\omega_1$ such that 
$\bigcup_{f\in \FF_{\alpha,\omega}} f^{``}\theta_\alpha=\omega_1$,
\item for all $\alpha<\beta<\gamma\le\omega$ we have that $\FF_{\alpha,\gamma}=\{f\circ g:\,g\in \FF_{\alpha,\beta}
\mbox{ and }f\in \FF_{\beta,\gamma}\}$,
\item $\FF_{\alpha,\alpha+1}$ always contains the identity function ${\rm id}_\alpha$ on $\theta_\alpha$ and
either this is all, or $\FF_{\alpha,\alpha+1}=\{{\rm id}_\alpha, h_\alpha\}$ for some $h_\alpha$ such that 
there is a {\em splitting point} $\beta$ with $h_\alpha\rest\beta={\rm id}_\alpha\rest\beta$ and
$h_\alpha(\beta)> \theta_\alpha$,
\item for every $\beta_0,\beta_1<\omega$ and $f_l\in \FF_{\beta_l,\omega}$ for $l<2$ there is $\gamma<\omega$
with $\beta_0,\beta_1<\gamma$, function $g\in \FF_{\gamma,\omega}$ and $f'_l\in \FF_{\beta_l,\gamma}$  such that $f_l
=g\circ f'_l$ for $l<2$.
\end{enumerate}
For future use in the proof we shall need the following:
\begin{observation}\label{increasing} The sequence $\langle \theta_\alpha\rangle_{\alpha<\omega}$ is a non-decreasing sequence diverging to $\infty$.
\end{observation}
To see this, notice that the fact that ${\rm id}_{\theta_\alpha}\in\FF_{\alpha,\alpha+1}$ guarantees that $\theta_\alpha\le
\theta_{\alpha+1}$. For the rest, suppose that $\langle \theta_{\alpha}\rangle_{\alpha<\omega}$
is bounded by some $k<\omega$. This implies that there is $\alpha<\omega$ such that for all $\gamma\ge \alpha$ we have
$\theta_{\gamma+1}=\theta_\gamma\le k$ and $\FF_{\gamma,\gamma+1}=\{{\rm id}_\gamma\}$. This implies by (4) that $\FF_{\alpha,\omega}=\{{\rm id}_\alpha\}$, in contradiction with (3). 

Fixing a simplified $(\omega,1)$-morass as in the above definitions,
we shall define a Boolean algebra $\alg$ as follows. By induction on $\alpha\le\omega$ we define a Boolean algebra 
$\alg_\alpha$ on a subset of $\theta_\alpha\cup (\omega_1\times\{0\})$ generated by  $\{i:\,i<\theta_\alpha\}$, and at the end we shall have $\alg=\alg_\omega$. The first requirement of the induction will be:
\begin{description}
\item[(i)] if $\beta<\alpha$ and $f\in \FF_{\beta,\alpha}$ then $f$ gives rise to a Boolean algebra embedding.
\end{description}
This requirement guarantees that the final object $\alg$ is indeed a well defined Boolean algebra. This is so because
$\{\alg_{\alpha}:\alpha\le\omega\}$ form an inverse system of Boolean algebras directed by the ordinal $\omega+1$, see 
\cite{Ha} for the details. For the main part of the proof,
suppose that we have Banach space $X_\ast$ in the ground model $\V$ with density $\aleph_1$ and a fixed dense set $\{z_i:\,i<\omega_1\}$ of $X_\ast$ in $\V$. 
We shall guarantee that for all natural numbers $n_*\ge 3$ and for all $j:\,\omega_1\to\omega_1$ there are $i_0, i_1 <\ldots < i_{n_*^2}$ such that:
\begin{equation*}\label{dichotomy}
\begin{split}
\mbox{ if }||z_{j(i_0)}+z_{j(i_1)}+\ldots z_{j(i_{n_*^2})}||<n_*-1&\mbox{, then } i_0, i_1,\ldots i_{n_*^2}\mbox{ are disjoint in }
\alg\mbox{ and }\\
 \mbox{ if }||z_{j(i_0)}+z_{j(i_1)}+\ldots z_{j(i_{n_*^2})}||\ge n_*-1 & \mbox{, then } i_0<_\alg  i_1\ldots <_\alg i_{n_*^2}. 
\end{split}\tag{$\ast$}
\end{equation*}
To see that our algebra, once constructed,
has the required properties, suppose that $T$ is an isomorphic embedding of $C({\rm St}(\alg))$ into some $X_\ast$ as above
and that $(\ast)$ holds for a dense set $\{z_i:\,i<\omega_1\}$ of $X_\ast$. Let
$x_i=T(\chi_{[i]})$ for $i<\omega_1$ and let $n_*\ge 3$ be large enough that for each $x\in C({\rm St}(\alg))$
we have that 
\[
\dfrac{1}{n_*}||x||< ||T(x)||<n_*||x||.
\]
For each $i<\omega_1$ let us choose $j(i)$ such that
$||x_i-z_{j(i)}||<\dfrac{1}{n_*^2+1}$. Let $i_0, i_1 <\ldots < i_{n_*^2}$ be as guaranteed by $(\ast)$. Suppose first that 
$||z_{j(i_0)}+z_{j(i_1)}+\ldots z_{j(i_{n_*^2})}||<n_*-1$. Then in the case $||z_{j(i_0)}+z_{j(i_1)}+\ldots z_{j(i_{n_*^2})}||<n_*-1$
we have that 
\[
||x_{i_0}+x_{i_1}+\ldots x_{i_{n_*^2}}||\le ||z_{j(i_0)}+z_{j(i_1)}+\ldots z_{j(i_{n_*^2})}||+\Sigma_{k\le n_*^2} ||x_{i_k}
-z_{j(i_k)}||< n_*-1+ \frac{n^2_*+1}{n^2_*+1}=n_*,
\]
yet $\dfrac{1}{n_*}||\chi_{[i_0]}+\ldots \chi_{[i_{n^2_\ast]}}||=\dfrac{n^2_\ast+1}{n_*}>n_\ast$, in contradiction with the choice
of $n_\ast$. The other case is similar.

Now we claim that to guarantee the condition $(\ast)$ for any fixed $X_\ast$ and $\{z_i:\,i<\omega_1\}$,
it suffices to assure that for all
\[
n_*\ge 3,
A\in [\omega_1]^{\omega_1}\in
\V , j_0:\,A\to\omega_1\in\V\mbox{  there are }i_0, i_1 <\ldots < i_{n_*^2}\in A \mbox{ exemplifying }(\ast).
\tag{$\ast\ast$}
\]
For this,
let us recall that for every $j:\,\omega_1\to\omega_1$ in $\V[G]$, there is $A\in [\omega_1]^{\omega_1}$ in 
$\V$ and $j_0:\,A\to\omega_1$ in $\V$ such that $j\rest A=j_0$. [Namely, if $\name{j}$ is a name for a function 
$j:\,\omega_1\to\omega_1$ then for every $\alpha<\omega_1$ there is $p\in G$ deciding the value $j(\alpha)$. Since the
forcing notion is countable, there is $p^\ast$ in $G$ such that the set $A$ of all $\alpha$ for which $p$ decides the
value of $\alpha$ is uncountable, and then it suffices to define $j_0$ on $A$ by $j_0(\alpha)=\beta$ iff $p$ forces
$j(\alpha)=\beta$.] In order to guarantee this, we first define by induction on $n$ an increasing sequence of natural numbers
$\alpha_n$ such that
$\alpha_0=0$ and $\theta_{\alpha_{n+1}}> 2\theta_{\alpha_n}+n$. In order to formulate
the second requirement of the inductive construction we need a lemma. 

\begin{lemma}\label{enoughspace} Let $n<\omega$. Then there is a set $A_n\subseteq \theta_{\alpha_{n+1}}$ of size $n+1$
such that for any $\beta\le \alpha_n$ and $f\in \FF_{\beta,\alpha_{n+1}}$, the intersection $A_n\cap f``\theta_\beta$
is empty.
\end{lemma}

\begin{Proof of the Lemma} Requirement (4) in the definition of the morass shows that it suffices to work with $\beta=\alpha_n$
as for $\beta<\alpha_n$ we have $\bigcup \{h``\theta_\beta:\,h\in \FF_{\beta,\alpha_{n+1}}\}\subseteq
 \{f``\theta_{\alpha_n}:\,f\in \FF_{\alpha_n,\alpha_{n+1}}\}$.
The proof proceeds on a backward induction on the size $k$ of $\alpha_{n+1}-\alpha_n\ge 1$. In the case $k=1$ we have that
$|\FF_{\alpha_n,\alpha_{n+1}}|\le 2$ and for each $f\in \FF_{\alpha_n,\alpha_{n+1}}$ we have $|f``\theta_{\alpha_n}|\le \theta_{\alpha_n}$ and therefore $\theta_{\alpha_{n+1}}\setminus |\bigcup_{f\in \FF_{\alpha_n,\alpha_{n+1}}}f``\theta_{\alpha_n}|\ge \theta_{\alpha_{n+1}}-
2\theta_{\alpha_n}>n$, so we can take
$A=\theta_{\alpha_{n+1}}\setminus \bigcup_{f\in \FF_{\alpha_n,\alpha_{n+1}}}f``\theta_{\alpha_n}$ . For the step $k+1$, let $\beta=\alpha_n+1$ and then it suffices to again apply (4) from the definition of a morass.
$\eop_{\ref{enoughspace}}$
\end{Proof of the Lemma}

Our second requirement of the induction will be as follows, where $r$ is
the generic Cohen real, viewed as a function from $\omega$ to 2 :
\begin{description}
\item[(ii)] Let $n<\omega$ and suppose that $\alg_{\alpha_n}$ has been defined.
Then $\alg_{\alpha_{n+1}}$ is generated freely over $\alg_{\alpha_n}$ except for the equations induced by requirement 
{\bf (i)} and the requirement that  $i_0, i_1,\ldots i_{n}\mbox{ are disjoint in }
\alg_{\alpha_{n+1}}$ if $r(n)=0$ and  $i_0 <_{\alg_{\alpha_{n+1}}}  i_1 <_{\alg_{\alpha_{n+1}}}\ldots <_{\alg_{\alpha_{n+1}}} i_{n}$ if 
$r(n)=1$, where $\{i_0, i_1, \ldots ,i_{n}\}$ is the increasing enumeration of the first $n+1$ elements of $A_n$. For $\beta\in (\alpha_n,\alpha_{n+1})$
we let $\alg_\beta$ be the Boolean span in $\alg_{\alpha_{n+1}}$ of $\theta_\beta$.
\end{description}
Given Lemma \ref{enoughspace}, it is clear that conditions {\bf (i)} and {\bf (ii)} can be met in a simple inductive construction of $\alg_a$, as there will be no possible contradiction between {\bf (i)} and the requirements that {\bf (ii)} puts on the elements of
$A_n$.  Let us now show that the resulting algebra is as required. 
Hence let us fix $n_\ast, X_\ast, \{z_i:\,i<\omega_1\}, A$ and $j_0$ as in $(\ast\ast)$. Is it is then sufficient to observe that
the following set is dense in the Cohen forcing (note that the set is in the ground model):
\[
\left\{p:\,(\exists n\in {\rm dom}(p))(n \ge  n_*^2\mbox{ and } p(n)=0\iff ||z_{j(i_0)}+z_{j(i_1)}+\ldots z_{j(i_{n_*^2})}||<n_*-1\right\},
\]
where $\{i_0, i_1, \ldots ,i_{n^2_*}\}$ is the increasing enumeration of the first $n_*^2+1$ elements of $A_n$.
$\eop_{\ref{Cohenreal}}$
\end{proof}

We note in the following Theorem \ref{Cohenrealc} that the construction from Theorem \ref{Cohenreal} can be refined so that the resulting
Banach space is weakly compactly generated. This give rise to Theorem \ref{wcgfinal} which shows that in the Cohen model
for $\aleph_2$ Cohen reals the universality number for wcg Banach spaces of density $\aleph_1$ is $\aleph_2$. A model
for this was obtained by Brech and Koszmider in \cite{BrKo2} using a ready-made forcing. Theorem \ref{wcgfinal} answers a question
raised by Koszmider (private communication). The reader not interested in wcg spaces can jump directly to after the
proof of Theorem \ref{Cohenrealc}, those who proceed to read Theorem \ref{Cohenrealc} will be assumed to know what wcg spaces are. Let us however recall one definition:

\begin{definition}\label{embeddingBell}
A subset $C$ of a
Boolean algebra $\alg$ has {\em the nice property}
if for no finite $F\subseteq C$ do we have $\bigvee F=1$.
A Boolean algebra $\alg$ is a c-{\em algebra} iff there is
a family $\{ B_n:\,n<\omega\}$ of pairwise disjoint antichains of $\alg$
whose union
has the nice property and generates $\alg$. 
\end{definition}

Bell showed in \cite{bell} that the Stone space of a c-algebra $\alg$ is a uniform Eberlein compact and hence $C({\rm St}(\alg))$
is a wcg space, see \cite{BrKo2} for details. We prove:

\begin{theorem}\label{Cohenrealc} 
Forcing with one Cohen real adds a c-algebra $\alg$ of size $\aleph_1$ such that
in the extension the Banach space $C({\rm St}(\alg))$ does not isomorphically embed into any space $C({\rm St}(\mathfrak B))$,
where $\mathfrak B$ is any Boolean algebra in
the ground model of size $\aleph_1$.
\end{theorem}

\begin{proof}
Fixing a simplified $(\omega,1)$-morass as in the proof of Theorem\ref{Cohenreal},
we shall define a Boolean algebra $\alg$ as follows. By induction on $\alpha\le\omega$ we define a Boolean algebra 
$\alg_\alpha$ on a subset of $\theta_\alpha\cup (\omega_1\times\{0\})$ generated by  $\{i:\,i<\theta_\alpha\}$ freely except for the
requirements in the next sentence, and at the end we shall have $\alg=\alg_\omega$. At the same time we shall define pairwise disjoint antichains $\langle B^n_\alpha:\,n<n_\alpha\rangle$ such that
$\bigcup_{n<n_\alpha}B^n_\alpha=\{i:\,i<\theta_\alpha\}$ and this set has the nice property. 
The basic requirements of the induction will be:
\begin{description}
\item[(i)] if $\beta<\alpha$ and $f\in \FF_{\beta,\alpha}$ then $f$ gives rise to a Boolean algebra embedding.
\item[(ii)] if $\beta<\alpha$ then $\theta_\alpha<\theta_\beta\implies n_\beta< n_\alpha$ and if $f\in \FF_{\beta,\alpha}$ and $i\in B^n_\alpha$ then $f(i)\in B^n_\beta$.
\end{description}
Requirement (i) guarantees that the final object $\alg$ is indeed a well defined Boolean algebra. The second requirement
guarantees that the algebra obtained is a c-algebra, as will become clear later. As in the proof of Theorem \ref{Cohenreal}
and using the same notation as there, we shall guarantee the requirement $(\ast\ast)$.

Our final requirement of the induction will be as follows, where $r$ is
the generic Cohen real, viewed as a function from $\omega$ to 2 :
\begin{description}
\item[(iii)] Let $n<\omega$ and suppose that $\alg_{\alpha_n}$ has been defined.
Then $\alg_{\alpha_{n+1}}$ is generated freely over $\alg_{\alpha_n}$ except for the equations induced by requirements
{\bf (i)} and {\bf (ii)} and the requirement that  $i_0, i_1,\ldots i_{n}\mbox{ are disjoint in }
\alg_{\alpha_{n+1}}$ if $r(n)=0$ and  $i_0 <_{\alg_{\alpha_{n+1}}}  i_1 <_{\alg_{\alpha_{n+1}}}\ldots <_{\alg_{\alpha_{n+1}}} i_{n}$ if 
$r(n)=1$, where $\{i_0, i_1, \ldots ,i_{n}\}$ is the increasing enumeration of the first $n+1$ elements of $A_n$. For $\beta\in (\alpha_n,\alpha_{n+1})$
we let $\alg_\beta$ be the Boolean span in $\alg_{\alpha_{n+1}}$ of $\theta_\beta$.
\end{description}
Given Lemma \ref{enoughspace}, it is clear that conditions {\bf (i)} and {\bf (ii)} can be met in a simple inductive construction of $\alg_a$, as there will be no possible contradiction between {\bf (i)} and the requirements that {\bf (ii)} puts on the elements of
$A_n$.  Let us comment on how to meet the requirement {\bf (iii)}, say at stage $n+1$. We have chosen $A_n$ so that it elements
are not in the image of any $B^m_\beta$ for any $\beta<\alpha_{n+1}$ and any $m<n_\beta$. Therefore we are free to 
either set $n_{\alpha_{n+1}}=n_\alpha+ 1$ and put all elements of $A_n$ into one antichain, if $r(n)=0$, or
set $n_{\alpha_{n+1}}=n_\alpha+ |A_n|$ and put each element of $A_n$ into a new antichain, requiring $i_0 <_{\alg_{\alpha_{n+1}}}  i_1 <_{\alg_{\alpha_{n+1}}}\ldots <_{\alg_{\alpha_{n+1}}} i_{n}$ if 
$r(n)=1$. There are no problems in generating $\alg_{\alpha+1}$ freely except for these requirements and those
dictated by (i)-(iii) as the requirements are not contradictory.

\begin{lemma}\label{c} Suppose that the requirements (i)-(iii)  from above are satisfied. Then
letting for $n<\omega$, $B^n=\bigcup_{\alpha\le\omega, n<n^\alpha}B^n_\alpha$, we obtain disjoint antichains which
have the nice property in $\alg$ and $\alg$ is generated by their union.
\end{lemma}

\begin{proof} (of the Lemma) As in the proof of Theorem \ref{Cohenreal}, (i) guarantees that $\alg$ is a Boolean algebra. Requirement (ii) guarantees that $B_n$s are disjoint antichains. The set $B=\bigcup_{n<\omega} B_n$
has the nice property since
this property is verified by finite subsets of $B_n$. In other words, 
if $B$ does not have the nice property, then there is a finite $F\subseteq B$ with $\vee F=1$. There must be $\alpha<\omega$ such that 
$F\subseteq \bigcup_{n<n_\alpha} B^n_\alpha$, so $\vee F=1$ already holds in $\alg_\alpha$. However, this is in contradictory
to the fact that $\alg_\alpha$ is generated freely except for the equations required in (i)-(iii). It is also clear $B$ generates $\alg$. 
$\eop_{\ref{c}}$
\end{proof} 
The proof that $(\ast\ast)$ holds is exactly the same as in the proof of Theorem \ref{Cohenreal}. 
$\eop_{\ref{Cohenrealc}}$
\end{proof}

We would now like to extend Theorems \ref{nonembedding} and \ref{Cohenreal} 
by using ``classical tricks" with iterations to obtain the non-existence of an isomorphically universal
Banach space of density $\lambda^+$ in an iterated extension. It turns out that
the cases $\lambda=\aleph_0$ and larger $\lambda$ are different. Let us consider the former case first.

{\bf Scenario.}
Let ${\mathbb P}$ denote ${\mathbb P}(\omega_1)$ and let ${\mathbb Q}$ be the iteration of $\omega_2$ steps of ${\mathbb P}$, say  over a model of GCH, or simply the classical extension to add $\aleph_2$ Cohen reals over a model
of $GCH$. Then in the extension by
${\mathbb Q}$ we have that $2^{\aleph_0}=\aleph_2$ and the new reals are added throughout the iteration.
We try to argue that the universality number for the class of Banach spaces of density $\aleph_1$ is $\aleph_2$, so 
we try for a contradiction. Suppose that the number is $\le\aleph_1$, then by Observation \ref{gap} there is actually a single universal element, say $X_\ast$. We can by Fact \ref{classicalf} assume that
$X_\ast=C({\rm St}(\alg_\ast))$ for some Boolean algebra $\alg_\ast$ of size
$\aleph_1$. By standard arguments, we can
find an intermediate universe in the iteration, call it ${\bf V}$, which contains
$\alg_\ast$. If we could say that $C({\rm St}(\alg_\ast))$ is also in ${\bf V}$, then we could apply Theorem
\ref{nonembedding} to conclude that for the generic Boolean algebra ${\mathfrak B}$ added to ${\bf V}$,
$C({\rm St}({\mathfrak B}))$ does not embed into $C({\rm St}(\alg_\ast))$ and we would be done. The problem
is that since we keep adding reals, $C({\rm St}(\alg_\ast))$ keeps changing from ${\bf V}$ to the final universe and
hence there is no contradiction.We resolve this difficulty by a use of simple functions with rational coefficients. \footnote{A preliminary
version of this proof was stated for isometric embeddings and upon hearing about it, Saharon Shelah suggested that
it should work for the isomorphic embeddings as well.}. 

\begin{theorem}\label{isometries} Let $\kappa$ be a cardinal with $\cf(\kappa)\ge\aleph_2$ and let $G$ be a generic for the iteration of length $\kappa$ with finite supports of the forcing to add the generic Boolean algebra
of size $\aleph_1$ by finite
conditions over a model $\V$ of GCH, or the forcing to add $\kappa$ many Cohen reals over a model $\V$ of GCH. Then in $\V[G]$ the universality number of the class of Banach spaces
of density $\aleph_1$ under {\em isomorphisms} is $\kappa$.
\end{theorem}

\begin{proof} For simplicity, let us show that there is no single universal element, the proof in the case of $<\kappa$ universals being very similar. Suppose that $C({\rm St}(\alg))$ is a universal Banach space of density $\aleph_1$
and without loss of generality, assume that $\alg$ is in the ground model.
Let $p^\ast$ force $\dot{T}$ to be an isomorphic embedding from the $C$ of the Stone space of
the first generic algebra to $C({\rm St}(\alg))$ with $||T||\le c< n_*$.  Let $\varepsilon>0$ be small enough, precisely $\varepsilon<\min\left\{\dfrac{n_*-c}{n*},
\dfrac{n^2_*+1-cn_*}{cn_*^2+1}\right\}$.
For $i<\omega_1$ let $p_i\ge p^\ast$
force that $h_i$ is a simple function with rational coefficients  (so $h_i\in \V$) satisfying $||\dot{T}(\chi_{[a_i]})- h_i||<\varepsilon$.
We proceed similarly to the proof of Theorem \ref{nonembedding}, and without loss of generality assume that $a_i\in {\rm dom}(p_i(0))$, for every $i$. Now 
do several $\Delta$-system arguments as standard, and take ``clean"  $i_0< i_1 <\ldots i_{n_*^2}$, hence $\bigcup_{k=0}^{n^2_*} p_{i_k}$ does not decide any relation between $a_{i_k}$s.  
If $||\Sigma_{k=0}^{n^2_*}h_{i_k}||\ge n^*$, it follows by the choice of $\varepsilon$ that there cannot
be an extension of $\bigcup_{k=0}^{n^2_*} p_{i_k}$ forcing $a_{i_k}$s and to be disjoint, a contradiction, and if  $||\Sigma_{k=0}^{n^2_*}h_{i_k}||<n^*$
then it follows that there cannot
be an extension of $\bigcup_{k=0}^{n^2_*} p_{i_k}$ forcing $a_{i_k}$s to be increasing with $k$, again a contradiction.
$\eop_{\ref{isometries}}$
\end{proof}

In the situation in which we add a generic Boolean algebra of cardinality $\lambda^+\ge\aleph_2$ with
conditions of size $<\lambda$, the forcing is countably closed and therefore it does not add any new reals. Hence the attempted
argument of proving the analogue of Theorem  \ref{isometries} for isomorphic embeddings works in this case.
However, to make this into a theorem we need an iteration theorem for the forcing, in order to make sure that the cardinals are preserved. Paper \cite{5authorsforcing} reviews various known forcing axioms and building up on them proves the following Theorem \ref{5authors}, which is perfectly suitable for our purposes.

\begin{theorem} [Cummings, D\v zamonja, Magidor, Morgan and Shelah] \label{5authors} Let $\lambda=\lambda^{<\lambda}$. Then, the iterations with $(<\lambda)$-supports
 of $(<\lambda)$-closed 
stationary $\lambda^+$-cc forcing which is countably parallel-closed, are $(<\lambda)$-closed 
stationary $\lambda^+$-cc.
\end{theorem}

Here, the property of countable parallel-closure is defined as follows:

\begin{definition}\label{parallel-closure} Two increasing sequences $\langle p_i:\,i<\omega\rangle$ and 
$\langle q_i:\,i<\omega\rangle$ of conditions in a forcing $\mathbb P$
are said to be {\em pointwise compatible} if for each $i<\omega$ the conditions $p_i, q_i$ are compatible.
The forcing $\mathbb P$ is said to be {\em countably parallel-closed} if for every two $\omega$-sequences of pointwise
compatible conditions as above, there is a common upper bound to $\{p_i, q_i:\,i<\omega\}$ in $\mathbb P$.
\end{definition}

\begin{lemma}\label{newmet} The forcing ${\mathbb P}(\lambda)$ is countably parallel-closed.
\end{lemma}

\begin{proof} Suppose that $\langle p_i:\,i<\omega\rangle$ and $\langle q_i:\,i<\omega\rangle$ are pairwise compatible increasing sequences. In particular this
means that, on the one hand, each $p_i, q_i$ agree on their intersections, and on the other hand that $\bigcup_{i<\omega} p_i$
and $\bigcup_{i<\omega} q_i$ each form a Boolean algebra. Furthermore $\bigcup_{i<\omega} p_i$ and 
$\bigcup_{i<\omega} q_i$ agree on their intersection, and hence their union can be used to generate a Boolean algebra, which will
then be a common upper bound to the two sequences.
$\eop_{\ref{newmet}}$
\end{proof}

\begin{theorem}\label{isomorphisms} Let  $\lambda=\lambda^{<\lambda}$, let $\cf(\kappa)\ge\lambda^{++}$ and let
$G$ be a generic for the iteration with $(<\lambda)$- supports of length $\kappa$ of the forcing to add the generic Boolean algebra
of size $\lambda^+$ by 
conditions of size $<\lambda$ over a model of GCH, or in the case of $\lambda=\aleph_0$ the forcing to add $\kappa$ many Cohen reals. Then in $\V[G]$:
\begin{description}
\item[(1)] 
the universality number of the class of Banach spaces
of density $\lambda^+$ under {\em isomorphisms} is $\kappa=2^{\lambda^+}$.
\item[(2)] 
there are $2^{\lambda^+}$ many pairwise non-isomorphic Banach spaces of density $\lambda^+$.
\end{description}
\end{theorem}

\begin{proof} (1) For the case of $\lambda=\aleph_0$ we use Theorem \ref{isometries}. In the case of uncountable $\lambda$, 
it follows from Lemma  \ref{properties}, Lemma \ref{newmet} and Theorem \ref{5authors} that the
iteration of the  forcing ${\mathbb P}(\lambda)$ described in the statement of the Theorem, preserves cardinals. By the 
countable closure of the forcing no reals are added, 
and hence the argument described in the {\bf Scenario}, with $\lambda^+$ in place of $\aleph_1$, works out to give the
desired conclusion.

{\noindent (2)} The claim is that the Banach spaces $C({\rm St}(\alg))$ added at the individual steps of the iteration
are not pairwise isomorphic. To prove this, we only need to notice that an embedding of a Banach space into another is determined by the restriction of
an embedding onto a dense set, hence in our case a set of size $\lambda^+$, and then to argue as in (1). 
$\eop_{\ref{isomorphisms}}$
\end{proof}

Using the same reasoning as in Theorem \ref{isomorphisms} with $\lambda=\aleph_0$ and relating this to Theorem
\ref{Cohenrealc} we obtainÉ

\begin{theorem}\label{wcgfinal} In the Cohen model for $2^{\aleph_0}=\kappa$, where ${\rm cf}(\kappa)\ge \aleph_2$, 
the universality number of the wcg Banach spaces is at least ${\rm cf}\kappa$ and there are ${\rm cf}(\kappa)$ pairwise
non-isomorphic wcg Banach spaces.
\end{theorem}

\section{Isometries and isomorphisms of small Banach-Mazur diameter}\label{smalldiameter}
As we have seen in \S\ref{Cohenlike}, the Cohen-like extensions do not distinguish the universality number of Banach spaces
of a fixed density under isomorphisms or under isometries, the subject of Conjecture \ref{conj2}. This short section gives evidence that some forms of that conjecture are just true in ZFC and serves as an interlude to motivate study of other kinds of embedding but isometries and isomorphisms. The results will be easy to a Banach space theorist as they use a well known theorem in the subject, due to Jarosz in \cite{Jarosz}. 

\begin{theorem}[Jarosz]\label{Jar} Suppose that $T:\,C(K)\to C(L)$ is an isomorphic embedding such that $||T||\cdot ||T^{-1}||<2$.
Then there is a closed subspace $L_1\subseteq L$ which can be continuously mapped onto $K$.
\end{theorem}

We can use Jarosz's theorem  to reduce the question of the existence
of isometrically universal Banach spaces or more generally, universals under isomorphisms of small 
Banach-Mazur diameter to that of Boolean algebras and their ideals. Let us recall the following 

\begin{observation}\label{isominto} Suppose that $L$ maps onto $K$, then $C(K)$ isometrically embeds into $C(L)$.
\end{observation}

\begin{proof} Let $F$ be a continuos mapping from $L$ that maps onto $K$. Defining for $f\in C(K)$ $T(f)$ by $Tf(x)=
f(F(x))$ for $x\in L$ gives the desired isometry.
$\eop_{\ref{isominto}}$
\end{proof}

We shall also use the following 

\begin{definition}\label{weak} Let $\MM$ ba a family of Boolean algebras and $\NN\subseteq \MM$. We say that
$\NN$ is {\em weakly universal} for $\MM$ if every element of $\MM$ embeds into a factor of an element of $\NN$.
\end{definition}

\begin{theorem}\label{thesame} For every $\kappa$, the following numbers are the same
\begin{itemize}
\item 
the universality  number of the class of Banach spaces of density $\kappa$ with isomorphic
embeddings of the Banach-Mazur diameter $<2$, 
\item the smallest size of a weakly universal 
family  of Boolean algebras of size $\kappa$
and
\item  the universality  number of the class of Banach spaces of 
density $\kappa$ under isometries.
\end{itemize}
\end{theorem}

\begin{proof} Let $\kappa$ be fixed. 
Suppose that $\{X_\alpha:\,\alpha<\alpha^\ast\}$ are universal under isomorphisms of the Banach-Mazur diameter $<2$
in the class of Banach spaces of density $\kappa$. By the introductory remarks, it follows that we can assume that for
every $\alpha$ there is a Boolean algebra $\alg_\alpha$ of size $\kappa$ such that $X_\alpha=C({\rm St}(\alg_\alpha))$.
The closed subspaces of ${\rm St}(\alg_\alpha)$ are of the form ${\rm St}(\alg_\alpha/I)$ where $I$ is an ideal in $\alg_\alpha$.
Let $\alg$ be a Boolean algebra of size $\kappa$ and let $K$ be it Stone space. Hence $C(K)$ embeds with a Banach-Mazur distance $<2$ in some $X_\alpha$. By Jarosz's theorem, there is a closed subspace $L$ of ${\rm St}(\alg_\alpha)$
such that $L$ maps continuously onto $K$. Let $I$ be an ideal on $\alg_\alpha$ such that $L={\rm St}(\alg_\alpha/I)$. It follows
that $\alg$ is a subalgebra of $\alg_\alpha/I$. Hence every Boolean algebra of size $\kappa$ embeds in a factor of some
$\{\alg_\alpha:\,\alpha<\alpha^\ast\}$.

On the other hand suppose that $\{\alg_\alpha:\,\alpha<\alpha^\ast\}$ is a family of Boolean algebras of size $\kappa$ 
such that every Boolean algebra of size $\kappa$  embeds in a factor of some $\alg_\alpha$. Let 
$X_\alpha=C({\rm St}(\alg_\alpha))$ and we claim that  $\{X_\alpha:\,\alpha<\alpha^\ast\}$ are universal under isometries in the class of Banach spaces of density $\kappa$. It suffices to show that for every Boolean algebra $\alg$ of size $\kappa$,
$C({\rm St}(\alg))$ isometrically embeds into some $X_\alpha$. Given such $\alg$, let $\alpha<\alpha^\ast$ be such that for
some ideal $I$ on $\alg$ we have that $\alg$ embeds into $\alg_\alpha/I$. Hence  ${\rm St}(\alg)$ is a continuous image of
the ${\rm St}(\alg_\alpha/I)$, which is a closed subspace $L$ of $X_\alpha$. By Observation \ref{isominto}, $C({\rm St}(\alg))$ embeds 
isometrically into $C(L)$. Let $B$ a basis of $C(L)$ as  a vector space. By Tietze's extension theorem, every function $f\in B$ can be extended to a function $g\in X_\alpha
=C({\rm St}(\alg_\alpha))$ of the same norm. Let this association be called $T$. Now extend $T$ to a linear embedding from $C(L)$
\footnote{of course, we are simply spelling out an application of the Hahn-Banach theorem}, which 
finally shows that $C(L)$ and hence $C(K)$ embeds isomorphically into $X_\alpha$. As the universality  number of the class of Banach spaces of 
density $\kappa$ under isometries is clearly $\ge$ the universality  number of the class of Banach spaces of density $\kappa$ with isomorphic
embeddings of the Banach-Mazur diameter $<2$, we obtain the desired equality.
$\eop_{\ref{thesame}}$
\end{proof}

%%%%%%%%%%%%%%%%%%%%%%%%%%%%%%%%%%%%%%%%
\section{Natural spaces of functions}\label{natural} As we have seen in \S \ref{smalldiameter},  there is a good motivation to study other kinds of embeddings of Banach spaces but isometries and isomorphisms. Sticking to the spaces of the form $C(K)$, among the classically studied isomorphic embeddings are those that preserve multiplication, or the ones that preserve the pointwise order of
functions. It is known for either one of them (Gelfand and Kolmogorov \cite{GelfKolm} for the former and Kaplansky \cite{Kaplansky} for the latter) that if they are onto
they actually characterize the topological structure of the space, that is if $T:\,C(K)\to C(L)$ is an onto embedding which either
preserves multiplication or the pointwise order, then $K$ and $L$ are homeomorphic. We shall show that in moving from the order
preserving onto assumption just a small bit, we no longer have the preservation of the homeomorphic structure, but under
the assumption that GCH fails sufficiently, we do have a large number of pairwise nonisomorphic spaces and a large universality
number.  The same methods can be applied to study embeddings with some amount of preservation of multiplication, which we
shall not do here. Our methods will involve a combination of model theory, set theory and Banach space theory. In this section we
introduce a simple model-theoretic structure which will be used to achieve that mixture of methods.

Suppose that $\alg$ is a Boolean algebra. We shall associate to it a simple structure whose role is
to represent the space $C(K)$, where $K$ is the Stone space of $\alg$, $K={\rm St}(\alg)$. The idea is as follows. 
We are interested in the set of all simple functions with rational coefficients defined on $K$, 
so functions of the type $\Sigma_{i\le n} q_i\chi_{[a_i]}$, where each $q_i$ is rational, $a_i\in \alg$ and
$[a_i]$ denotes the basic clopen set in $K$ determined by $a_i$.
Every element of
$C(K)$ is a limit of a sequence of such functions, since the limits of such sequences form exactly the class
of Lebesgue integrable functions, which of course includes $C(K)$.  Let us then consider the vector space freely generated by $\alg$ over ${\mathbb Q}$, call it $V=V(\alg)$.
\footnote{This vector space figures in \cite{BrKo} with the notation $C_{\mathbb Q}(\alg)$ and is considered
in a different context.}
Hence every simple function on $K$ with rational coefficients corresponds uniquely to an element of
$V$, via an identification of each $a\in \alg$ with $\chi_{[a]}$. Using coordinatwise addition and scalar
multiplication the product $W=V^\omega$ becomes a vector space.
Any function $f$ in $C(K)$ can be identified with an element of this vector space, namely a sequence of
simple rational functions whose limit is $f$, and hence $C(K)$ can be identified with a subset of $W$.

To encapsulate this discussion we shall work with vector spaces with rational coefficients and with two
distinguished unary predicates $C$, $C_0$ satisfying $C_0\subseteq C$.
With our motivation in mind, we shall call them {\em function spaces}.
If such a space $(V,C, C_0)$ is the space of sequences of simple rational functions over a Stone space 
$K={\rm St}(\alg)$ and 
$C, C_0$
correspond respectively to the set of such sequences which converge or converge to 0, then we call $(V,C, C_0)$ a {\em natural space}
and we denote it by $N(\alg)$. In spaces of the form $N(\alg)$ for an element $\bar{f}$ of $C^{N(\alg)}$ we define
$||\bar{f}||$ as the norm in $C({\rm St}(\mathfrak A))$ of the limit $f$ of $\bar{f}$. If $\phi$ is an embedding 
between $N(\alg)$ and $N(\mathfrak B)$ we shall say that $D>0$ is a constant of the embedding if for every 
$\bar{f}$ of $C^{N(\alg)}$ we have that $ \dfrac{1}{D} \cdot ||\bar{f}||\le  ||\phi(\bar{f})||\le D\cdot  ||\bar{f}||$. Not every embedding
has such a constant, but we shall only work with the ones which do.

We shall mostly be interested in a specific case of
the representation of continuous functions as limits of simple functions, given by the following observation:

\begin{lemma}\label{inI} Suppose that $K={\rm St}(\alg)$ is the Stone space of a Boolean algebra $\alg$ and let $f\ge 0$
be a function in $C(K)$ with $||f||\le D^\ast$ for some $D^\ast>0$. Then there is a sequence $\langle f_n:\,n<\omega\rangle$ of 
simple functions, where each $f_n$ is of the form 
$\Sigma_{i\le n} q_i\chi_{[a_i]}$, with each $q_i$ rational in $(0,D^\ast]$ and $a_i\in \alg$, such that
$f=\lim f_n$.
\end{lemma}

\begin{proof} By multiplying by a constant if necessary, we can assume that $D^\ast=1$. Functions of the form 
$\Sigma_{i\le n} r_i\chi_{[a_i]}$, with each $r_i$ is real, contain the constant function 1, form an algebra and separate the
points of $K$, hence by the Stone-Weierstrass theorem they form a dense subset of $C(K)$. Notice that every function 
$\Sigma_{i\le n} r_i\chi_{[a_i]}$ can be, by changing the coefficients and the sets $a_i$ if necessary, represented
in the form where all $[a_i]$s are pairwise disjoint, so we can without loss of generality work only with such functions.
 Given $\varepsilon>0$
and $\Sigma_{i\le n} r_i\chi_{[a_i]}$ with $a_i$s disjoint, we can find for $i\le n$ rational numbers $q_i$ with $|q_i-r_i|<\varepsilon$, hence
the function $\Sigma_{i\le n} r_i\chi_{[a_i]}$ is approximated within $\varepsilon$ by $\Sigma_{i\le n} q_i\chi_{[a_i]}$,
showing that also functions with rational coefficients and sijoint $a_i$s are dense. Now given $f\ge 0$
a function in $C(K)$ with $||f||\le 1$ and $\varepsilon>0$, let $\Sigma_{i\le n} q_i\chi_{[a_i]}$ be a function with
rational coefficients and disjoint $a_i$s satisfying $||f- \Sigma_{i\le n} q_i\chi_{[a_i]}||< \varepsilon$, recalling that the $||\, ||$
in $C(K)$ is the supremum norm. Define now:
\[
s_i =\begin{cases}
		q_i  & \mbox{if } q_i\in [0,1],\\
		0 & \mbox{if } q_i < 0,\\
		1 & \mbox{if } q_i >1,
	\end{cases}
\]
and consider the function $\Sigma_{i\le n} s_i\chi_{[a_i]}$. We claim that $||f- \Sigma_{i\le n} s_i\chi_{[a_i]}||< \varepsilon$.
By the assumptions that $a_i$s are disjoint, for any $x$ there is at most one $i=i(x)$ such that $x\in [a_i]$.
If $x\notin\bigcup_{i\le n} [a_i]$ or $s_{i(x)}=q_{i(x)}$ then $|f(x)- \Sigma_{i\le n} s_i\chi_{[a_i]}(x)|= |f(x)-\Sigma_{i\le n} q_i\chi_{[a_i]}(x)|<
\varepsilon$. If $x\in [a_i]$ and $q_i<0$ then $|f(x)-s_i|=f(x)<f(x)-q_i< \varepsilon$ as $f(x)\ge 0$. If $x\in [a_i]$ and $q_i>1$ then 
$|f(x)-s_i|=1-f(x)<q_i- f(x)< \varepsilon$ as $f(x)\le 1$, which finishes the proof.
$\eop_{\ref{inI}}$
\end{proof}

\begin{definition}\label{topping} Suppose that $\bar{f}=\langle f_n:\,n>\omega\rangle$ is a sequence in $N(\alg)$ and suppose
that $\bar{f}'=\langle f'_n:\,n>\omega\rangle$ was obtained by first replacing each $f_n$ with an equivalent function 
$\Sigma_{i\le n} q_i\chi_{[a_i]}$ with disjoint $a_i$s and then replacing the coefficients $q_i$ by $s_i$ using the
procedure described in the proof of Lemma \ref{inI}. We say that $\bar{f'}$ is a {\em top-up} of $\bar{f}$.
\end{definition}

\begin{corollary}\label{spanning} Suppose that $\alg$ and $\mathfrak B$ are Boolean algebras and let ${\mathcal A}$ denote the linear subspace
of $C^{N(\alg)}$ spanned by the functions whose rational coefficients are in $[0,1]$. Then for every $\bar{f}=\langle f_n:\,n<\omega\rangle$
in $C^{N(\alg)}$ there is $\bar{g}\in {\mathcal A}$ with $\lim_n f_n=\lim_n g_n$.
\end{corollary} 

\begin{proof} Let $f=\lim_n f_n$, hence $f$ can be written as $f=f^+-f^{-}$ where $f^+=\max\{f,0\}$ and $f^-=\min\{f, 0\}$
are both continuous and positive. Therefore, by the closure of ${\mathcal A}$
under linear combinations  it suffices to prove the Corollary in the case of $f\ge 0$. Let $D^\ast=||f||$,
and we now  apply Lemma \ref{inI}.
$\eop_{\ref{spanning}}$
\end{proof}

The point of these definitions is the connection between the embeddability in the class of spaces of the form
$C({\rm St}(\alg))$ and the class of function spaces. Namely, we have the following

\begin{theorem}\label{induced} 
Suppose that $\alg$ and $\mathfrak B$ are Boolean algebras and let ${\mathcal A}$ denote the linear subspace
of $C^{N(\alg)}$ spanned by the functions whose rational coefficients are in $[0,1]$. Then
if there is an 
isomorphic embedding $T$ from $C({\rm St}(\alg))$ to $C({\rm St}(\mathfrak B))$  then
there is an isomorphic embedding $\phi$ from ${\mathcal A}$ to $N(\mathfrak B)$ satisfying that for every 
$\bar{f}=\langle f_n:\,n<\omega\rangle$ in ${\mathcal A}$, if $f=\lim_{n\in \omega} f_n$, then $\lim_{n\in \omega}\phi(\bar f)_n=
T(f)$.
\end{theorem}

\begin{proof} 
Let $T:\,C({\rm St}(\alg))\to C({\rm St}(\mathfrak B))$ be an isomorphic embedding, so
$||T||<\infty$.
We intend to define an isomorphic embedding $\phi$ from ${\mathcal A}$ to $N(\mathfrak B)$. By linearity it is sufficient to
work with the basis of ${\mathcal A}$, so denote by ${\mathcal A}'$
the sequences of simple rational functions whose coefficients are in $[0,1]$. Let us use the notation
$\pi_n$ for the projection on the $n$-th coordinate.
First we define the action of $\phi$ on those $f\in {\mathcal A}'$ which have the property that there
is at most one $n$ such that $\pi_n(f)$ is not the identity zero function, and then $\pi_n(f)$ is
a function of the form $\chi_{[a]}$ for some $a\in\alg$.
If there is no such $n$ with $a\neq 0$ then we let $\phi(f)$ be the element of $N(\mathfrak B)$
whose all projections $\pi_n$ are zero. Otherwise, let $n$ be such that $\pi_n(f)\neq 0$ and consider
$T(\pi_n(f))$, which is well defined.
We have no reason to believe that $T(\pi_n(f))$ is a simple function with rational coefficients.
However, there is a function $F((\pi_n(f)))$ which is a simple function with rational coefficients
% in [-1,1] 
and whose
distance to $T(\pi_n(f))$ in $C({\rm St}(\mathfrak B))$ is less than $\dfrac{1}{2^{n+1}}$. We define
$\phi(f)$ to be the unique element $g$ of $N(\mathfrak B)$ such that the only $\pi_m(g)$
which is not identically zero is $\pi_n(g)$ and $\pi_n(g)=F((\pi_n(f)))$. 

Now suppose that $\bar{f}\in {\mathcal A}'$ is such that for exactly one $n$, $\pi_n(\bar{f})$ is not identity
zero and $\pi_n(\bar{f})=\sum_{i=0}^m q_i\chi_{[a_i]}$ for some $a_0,\ldots a_m\in \alg$ and
some rational $q_0,\ldots q_{m}$ in $[0,1]$. For $i\le m$ let $\bar{f}_i$ be the element of $N(\alg)$ whose 
$n$-th projection is $\chi_{[a_i]}$ and all other projections are identity zero. Hence we have
already defined $\phi(\bar{f}_i)$ and we let $\phi(\bar{f})=\sum_{i\le m} q_i\phi(\bar{f}_i)$.

Finally suppose that $f=\langle f_n:\,n<\omega\rangle $ is any  element of ${\mathcal A}'$. Therefore
for every $n$ we have already defined $g_n=\phi(\langle 0,\ldots, f_n, 0\ldots\rangle))$, where $f_n$ is 
on the $n$-th coordinate. Let $\phi(f)=\langle g_n:\,n<\omega\rangle$.
Hence we have defined a linear embedding of  ${\mathcal A}'$ to $N({\mathfrak B})$. We extend this embedding to
${\mathcal A}$ by linearity. We need
to check that this embedding preserves $C$ and $C_0$. 
So suppose that $\bar{f}=\langle f_n:\,n<\omega\rangle$
is in $C^{N(\alg)}\cap {\mathcal A}$ and let $\bar{g}=\langle g_n:\,n<\omega\rangle$ be its image under $\phi$.
We shall show that $\bar{g}\in
C^{N({\mathfrak B})}$ by showing that it is a Cauchy sequence. Let $n, m<\omega$, we shall
consider $||g_n-g_m||$. Clearly it suffices to assume that $\bar{f}\in {\mathcal A}'$.
Let $f_n=\sum_{i\le k} q_i\cdot \chi_{[a_i]}$ and 
$f_m=\sum_{j\le l} r_j\cdot \chi_{[b_j]}$. We have 
\[
||g_n-g_m||=||\phi(f_n)-\phi(f_m)||\le ||\phi(f_n)-T(f_n)||+||T(f_n)- T(f_m)||+||T(f_m)-\phi(f_m)||
\]
\[
\le \sum_{i\le k} |q_i|||\phi(\chi_{[a_i]})-T(\chi_{[a_i]})||+||T(f_n-f_m)||+
\sum_{j\le l} |r_j|||\phi(\chi_{[b_j]})-T(\chi_{[b_j]})||
\]
\[
\le \dfrac{(k+1)}{2^{n+1}}+||T||\cdot ||f_n-f_m||+\dfrac{(l+1)}{2^{m+1}},
\]
which goes to 0 as $n,m\to\infty$.

At the end suppose that 
$\bar{f}=\langle f_n:\,n<\omega\rangle$
is in $C_0^{N(\alg)}$
and let $\bar{g}=\langle g_n:\,n<\omega\rangle$ be its image under $\phi$. 
By the definition of $\phi$ we have that $||g_n||\le ||T(f_n)||+\dfrac{(n+1)}{2^{n+1}}$.
Since $||T||<\infty$ we have that $\lim_{n\to \infty}||T(f_n)||=0$, so in conclusion,
$\lim_{n\to \infty}||g_n||=0$.
$\eop_{\ref{induced}}$
\end{proof}

\section{Invariants for the natural spaces and very positive embeddings}\label{vpemb} We shall now adapt the Kojman-Shelah method of invariants \cite{KjSh409},
to the natural spaces and a specific kind of isomorphic embeddings between Banach spaces, which we call very positive embeddings (see Definition \ref{strictlypositive}). From this point on we assume that $\lambda$ is a regular uncountable cardinal.

\begin{definition} (1) Suppose that $ M$ is a model of size
$\lambda$. A {\em  filtration} of $M$ is a continuous increasing sequence $\langle M_\alpha:\,\alpha<\lambda\rangle$ 
of elementary submodels of $M$, each of size $<\lambda$.

{\noindent (2)} For a regular cardinal $\theta<\lambda$ we use the notation $S^\lambda_\theta$ for $\{\alpha<\lambda:\,{\rm cf}(\alpha)=\theta\}$.

{\noindent (3)} A {\em club guessing sequence on $S^\lambda_\theta$} is a sequence $\langle C_\delta:\,\delta\in S^\lambda_\theta\rangle$ such that 
each $C_\delta$ is a club in $\delta$, and for every club $E\subseteq \lambda$ there is $\delta$ such that  $C_\delta
\subseteq E$.
\end{definition}

\begin{observation}\label{cofinalities} Suppose that $\theta>\aleph_1$ and there is a club guessing sequence 
$\langle C_\delta:\,\delta\in S^\lambda_\theta\rangle$. Then there is a club guessing sequence 
$\langle D_\delta:\,\delta\in S^\lambda_\theta\rangle$  such that for all $i<\theta$
\begin{equation}\label{cofomega}
\cf(i)\neq \omega\implies  \cf(\alpha^\delta_{i})\neq\omega,
\end{equation}
where $\langle \alpha^\delta_i:\,i<\theta\rangle$ is the increasing enumeration of $D_\delta$, for each $\delta$.
\end{observation}

\begin{proof} First of all notice that by passing to subsets if necessary we can without loss of generality assume
that each $C_\delta$ has order type $\theta$. Given $\delta$, let $C'_\delta$ consists of the points of $C_\delta$ of cofinality $>\omega$ and let
$D_\delta$ be the closure of $C'_\delta$ in $\delta$. Since $\theta>\aleph_1$ we have that
$C'_\delta$ is unbounded in $\delta$, so it is clear that $D_\delta$ is
a club of $\delta$ and since we have $D_\delta\subseteq C_\delta$, we obtain that the resulting
sequence is a club guessing sequence on $S^\lambda_\theta$. It also follows that ${\rm otp}(D_\delta)=\theta$, so the
increasing enumeration as claimed exists.
$\eop_{\ref{cofinalities}}$
\end{proof}

The main definition we need is the definition of the {\em invariant}.
Let us suppose that $\theta>\aleph_1$ is regular and that $\langle D_\delta:\,\delta\in S^\lambda_\theta\rangle$ is a club guessing
sequence with an increasing enumeration $\langle \alpha^\delta_i:\,i<\theta\rangle$ of $D_\delta$, for each $\delta$ and satisfying
the requirement (\ref{cofomega}). This sequence will be fixed throughout. The existence of such a sequence will be discussed
at the end of the section but for the moment let us say that Shelah (see Theorem \ref{guess square}) proved that such a sequence exist in many circumstances, notably
for any $\lambda$ regular $\ge \theta^{++}$.

\begin{definition} Suppose that $\mathfrak A$ is a Boolean algebra of size $\lambda$, $\bar{\mathfrak A}=\langle {\mathfrak A}_\alpha:\,\alpha <
\lambda\rangle$ a filtration of $\mathfrak A$, that $\delta\in S^\lambda_\theta$ and that $\bar{f}\in C^{N(\mathfrak A)}
\setminus N(\mathfrak A_\delta)$. An ordinal $i\in S^\theta_{\neq \omega}$ is an element of {\em the invariant} ${\rm inv}_{\bar{\mathfrak A},\delta}(\bar{f})$ iff there is
$\bar{f}'\in C^{N({\mathfrak A}_{\alpha^\delta_{i+1}})}$ 
such that for every $\bar{g}$ in $C^{N(\mathfrak A_{\alpha^\delta_i})}$
we have
\begin{equation}
0\le \lim_n g_n\le |\lim_n f_n-\lim_n f'_n| \implies \bar{g}\in C_0.
\label{eqn:starjedan}
\end{equation} 
\end{definition}

We shall be interested in the kind of embeddings between Banach spaces which will allow us to define appropriate $\phi$
which preserve the invariants, see the Preservation Lemma \ref{preservation}. We have succeeded to do this in the case of
a special kind of positive embeddings, as defined in the following definition.

\begin{definition}\label{strictlypositive} We say that an isomorphic embedding $T:\,C(K)\to C(L)$ is
{\em very positive} if the following requirements hold:
\begin{description}
\item{(i)} $g\ge 0\implies Tg\ge 0$ (positivity),
\item{(ii)}  for every $g\in C(L)\setminus\{0\}$ with $0\le g$, there is $h\ge  0$ with $0\le Th\le g$ and $h\neq 0$,
\item{(iii)} if $0\le Th\le Tf$, $h,f\neq 0$ and $h\neq f\ge 0$ then there is $s\ge 0$ definable from $h$ with $0\le s\le f$.
\end{description}
\end{definition}

We do not know if very positive embeddings were studied in the literature but clearly, one kind of embedding that is very positive, is an order preserving onto embedding. In this case we have Kaplansky's theorem \cite{Kaplansky} mentioned above, which shows that in the presence
of such an embedding from $C(K)$ to $C(L)$ we have that $K$ and $L$ are homeomorphic. We show in the example in
\S \ref{example} that 
the analogue is not true for very positive embeddings, not even when they are assumed to be onto. In particular the question of the number of pairwise nonisomorphic
by very positive embeddings spaces of the form $C({\rm St}(\mathfrak A))$ does not reduce to the well studied and
understood question of the number of pairwise nonisomorphic Boolean algebras of a given cardinality (which for
any infinite $\kappa$  is
always equal to $2^\kappa$, see Shelah's \cite{Sh-c}).  

Let us now make a further assumption on $\lambda$:
\begin{equation}
\kappa<\lambda\implies \kappa^{\aleph_0}<\lambda.
\label{eqn:star}
\end{equation}

\begin{lemma} (Preservation Lemma)\label{preservation} Let $\mathfrak A$ and $\mathfrak B$ be Boolean algebras of size $\lambda$ and 
suppose that $T:\,C({\rm St} (\mathfrak A))\to C({\rm St} (\mathfrak B))$ is a very positive embedding.
Let $\bar{\mathfrak A}$ and $\bar{\mathfrak B}$ be any filtrations of $\mathfrak A$ and $\mathfrak B$ respectively and let 
${\mathcal A}$ denote the linear subspace
of $C^{N(\alg)}$ spanned by the set ${\mathcal A}'$ of sequences of functions whose rational coefficients are in $[0,1]$. 

If
$\phi: {\mathcal A} \to N(\mathfrak B)$ is an isomorphic embedding satisfying that $\lim\phi(\bar{f})=T(\lim(\bar{f}))$ for every
$\bar{f}\in {\mathcal A}$, 
then there is
a club $E$ of $\lambda$ such that for every $\delta$ with $D_\delta\subseteq E$ and for every  $\bar{f}\in {\mathcal A}'
\setminus N(\mathfrak A_\delta)$ with $0\le \lim_n f_n $ and $||\lim_n f_n||=1$
we have that 
\[
{\rm inv}_{\bar{\mathfrak A},\delta}(\bar{f})={\rm inv}_{\bar{\mathfrak B},\delta}(\phi(\bar{f})).
\]
\end{lemma}

\begin{proof}
We may assume that the underlying set of $\mathfrak A$ and $\mathfrak B$ is the ordinal $\lambda$.
Let us define a model $M$ with the universe two disjoint copies of the $\omega$-sequences of the simple functions
on $\lambda$ with rational coefficients, interpreted as
the elements of $N(\mathfrak A)$ and 
$N(\mathfrak B)$, all the symbols of  $N(\mathfrak A)$ and $N(\mathfrak B)$ with interpretations induced from these models, and the symbols $\mathcal A$, $\mathcal A'$ and $\phi$. 
By the assumption (\ref{eqn:star}), there is a club $E$ of $\lambda$ such that for every $\delta\in E$ of cofinality not $\omega$
we have that $M$ restricted to the sequences whose ordinal coefficients are $<\delta$ is an elementary submodel of $M$ and that
is has universe corresponding to $N({\mathfrak A}_\delta)\cup N({\mathfrak B}_\delta)$. Let us denote the latter model by $M\rest\delta$.

Suppose now that $D_\delta\subseteq E$.
Choose $\bar{f}\in C^{N(\mathfrak A)}\cap {\mathcal A'} \setminus N({\mathfrak A}_\delta)$ 
with $0\le \lim_n f_n$ and $||\lim_n f_n||=1$. Let $f=\lim f_n$. By the choice of $\delta$ we have that 
$\phi(\bar{f})\in C^{N(\mathfrak B)}\setminus N({\mathfrak B}_\delta)$. Suppose first that $i\in {\rm inv}_{\bar{\mathfrak A},\delta}(\bar{f})$
and let $\bar{f}'\in C^{N({\mathfrak A}_{\alpha^\delta_{i+1}})}$ demonstrate this. 
Let $f'\deq \lim f'_n$,
which is well defined as $\bar{f}'\in C^{N({\mathfrak A})}$. 
Notice that the requirement (\ref{eqn:starjedan}) will hold if we replace
$\bar{f'}$ by any top-up $\bar{f''}$ (see Definition \ref{topping}) as $0\le f\le 1$
and hence for all $x$ we have $|f(x)-f''(x)|\le
|f(x)-f'(x)|$. Since the topping up procedure is definable in $M\rest\alpha^\delta_{i+1}$, we may assume
that $0\le  f'_n \le 1$ for all $n$ and $0\le f'\le 1$.
By Lemma \ref{inI} applied within
$M_{\alpha^\delta_{i+1}}$ we can assume that $\bar{f}'\in{\mathcal A}$.  By the choice of $\phi$ we have that 
$\lim \phi(\bar{f}')=T(f')$
and similarly $0\le \lim \phi(\bar{f}')$. 
By the fact that $D_\delta\subseteq E$ and since $\phi$ is an isomorphism we have 
that $\phi(\bar{f}')\in C^{N({\mathfrak B}_{\alpha^\delta_{i+1}})}$. We would like to use $\phi(\bar{f}')$ to witness that 
$i\in {\rm inv}_{\bar{\mathfrak B},\delta}(\phi(\bar{f}))$, so let us try. By the choice of $\phi$ we have that 
$\lim \phi(\bar{f}')=T(f')$
and similarly $0\le \lim \phi(\bar{f}')$. It remains to check the property (\ref{eqn:starjedan}) of $\phi(\bar{f}')$. 

Suppose for a contradiction that there is $\bar{g}\in 
N({\mathfrak B}_{\alpha^\delta_{i}})$ such that $0\le g\deq \lim_n g_n\le |Tf-Tf'|$ but that $\bar{g}\notin C_0$. 
Applying (ii) we can find $h\ge 0$ with $0\le Th\le g$ and $h\neq 0$. By Corollary \ref{spanning} we can assume that there is
$\bar{h}\in {\mathcal A}$
with $h=\lim_n h_n$ and hence $Th=\lim_n \phi (\bar{h})$. Translating (ii) into the terms of $\phi$ and applying the elementarity of $M\rest \alpha^\delta_i$ we can assume
that  $\bar{h}\in N(\alg_{\alpha^\delta_i})$. Now we apply (iii) to find $s\ge 0$, $s\neq 0$ definable from $h$ and satisfying 
$s\le |f-f'|$. Being definable from $h$, $s$ has an approximation $\bar{s}$ with $\bar{s}$ definable from $\bar{h}$,
hence $\bar{s}\in N(\alg_{\alpha^\delta_i})$.
By topping up if necessary as in Lemma \ref{inI} and in the above paragraph,
we can assume that every element $s_n$ in $\bar{s}$ satisfies $s_n\ge 0$, therefore $\bar{s}$ contradicts the choice
of $\bar{f}'$.

Now let us prove the other direction of the desired equality. Let $i\in {\rm inv} _{\bar{\mathfrak B},\delta}(\phi(\bar{f}))$ as
exemplified by some $\bar{g}\in N({\mathfrak B}_{\alpha^\delta_{i+1}})$.  As in the previous paragraphs, we can
assume that 
$0\le g\deq \lim_n g_n$ and hence by (ii) we can assume that for some $f'\ge 0$ we have
$0\le Tf'\le g$ and by the same argument as above we can assume that there is $\bar{f}'=\langle f'_n:\,n<\omega\rangle\in N(\alg_{\alpha^\delta_{i+1}})\cap {\mathcal A}$
such that $\lim_n f'_n=f'$. Now we claim that $\bar{f}'$ exemplifies that $i\in 
{\rm inv} _{\bar{\mathfrak A},\delta}(\bar{f})$. Suppose for a contradiction that $\langle h_n:\,n\in \omega\rangle\in 
N(\alg_{\alpha^\delta_i})\setminus C_0$ and $0\le \lim h_n\le |\lim f_n-\lim f'_n|$. Let $h=\lim h_n$. As before,
we can assume that $\bar{h}\in {\mathcal A}$ and each $h_n\ge 0$. So by the positivity we
have $0\le Th\le T(|f-f'|)=|Tf-Tf'|$, by the choice of $\phi$ we have that $\phi(\bar{h})=Th$, by elementarity we have
$\phi(\bar{h})\in N({\mathfrak B}_{\alpha^\delta_i})\setminus C_0$. It follows that $\phi(\bar{h})$
contradicts $i\in {\rm inv} _{\bar{\mathfrak B},\delta}(\phi(\bar{f}))$.
$\eop_{\ref{preservation}}$
\end{proof}

The next task is to construct lots of Boolean algebras $\alg$ with different invariants for $N(\alg)$ and then to us the Preservation 
Lemma to show that no fixed $N(\mathfrak B)$ can embed them all.

\begin{lemma} (Construction Lemma)\label{construction} Suppose that $\theta^+<\lambda$.
Then the club guessing sequence $\langle D_\delta:\,\delta\in S^\lambda_\theta\rangle$ can be chosen so that
for any $A\subseteq\theta$ which is
a closed set of limit ordinals, there is a Boolean
algebra $\alg=\alg[A]$, a filtration $\bar{\alg}$ of $\alg$ and a club $E$ of $\lambda$ such that for every $\delta\in E$ there is
$\bar{f}\in {\mathcal A}' \setminus N(\alg_\delta)$ with ${\rm inv}_{\bar{\alg}, \delta}(\bar{f})=A\cap S^\theta_{\neq\omega}$ and
$||\lim f_n||=1$.
\end{lemma}

The proof of this lemma is presented in the \S \ref{proofCL}. The following theorem of Shelah will be used in the proof of the Construction Lemma as well as in the
the proof of Theorem \ref{glavni}.

\begin{theorem}[Shelah, \cite{Sh420}, 1.4.]\label{guess square}
Let $\theta<\lambda$ be two regular cardinals with $\theta^+<\lambda$.
Then
there is a stationary set $S\subseteq S^\lambda_\theta$ and sequences $\langle c_\delta:\,\delta\in S\rangle$, 
$\langle \PP_\alpha:\,\alpha<\lambda\rangle$ such that:
\begin{itemize}
\item ${\rm otp}(c_\delta)=\theta$ and $\sup(c_\delta)=\delta$,
\item for every club $E$ of $\lambda$ there is $\delta\in S$ with $c_\delta\subseteq E$,
\item $\PP_\alpha\subseteq \PP(\alpha)$ and $|\PP_\alpha|<\lambda$,
\item if $\alpha$ is a non-accumulation point of $c_\delta$ then $c_\delta\cap\alpha\in \bigcup_{\alpha'<\alpha} \PP_\alpha'$.
\item the non-accumulation points of every $c_\delta$ are successor ordinals\footnote{Claim 1.4. in \cite{Sh420} does not state this
property explicitly, but it follows from the first line of the proof of that Claim.}.
\end{itemize}

\end{theorem}

Now we present the main theorem of this part of the paper.

\begin{theorem}\label{glavni} Suppose that $\theta$ and $\lambda$ are two regular cardinals with
$\aleph_2\le\theta<\theta^+<\lambda$, and that $(\forall \kappa<\lambda) \kappa^{\aleph_0}<\lambda$.

Then
\begin{description}
\item[(1)]
the minimal number of spaces of the form $C({\rm St} (\alg))$ of density $\lambda$ needed to embed all Banach spaces 
of the form $C({\rm St} ({\mathfrak B}))$ of density $\lambda$
very positively is $2^\theta$. In particular, if $2^\theta>\lambda$ then there is no very-positively universal 
space $C({\rm St} (\alg))$ of density $\lambda$.
\item[(2)]
if $2^\theta>\lambda$ then there are at least ${\rm cf}(2^\theta)$ pairwise non-very positively isomorphic Banach spaces $C(K)$ of
density $\lambda$.
\end{description}
\end{theorem}

\begin{proof} Fix sequences $\langle c_\delta:\,\delta\in S\subseteq S^\lambda_\theta\rangle$ and $\langle \PP_\alpha:\,\alpha
<\lambda\rangle$ as guaranteed by Theorem \ref{guess square}. Notice that $\langle c_\delta:\,\delta\in S\rangle$ satisfies that
with $\langle \alpha^\delta_i:\,i<\theta\rangle$ being the increasing enumeration of $c_\delta$, we have that $\cf(i)\neq\omega
\implies \cf(\alpha^\delta_i)\neq\omega$. Hence letting $D_\delta=c_\delta$ for $\delta\in S$ and
$D_\delta$ an arbitrary club of $\delta$ of order type $\theta$ satisfying 
$\cf(i)\neq\omega
\implies \cf(\alpha^\delta_i)\neq\omega$ for $\delta\in S^\lambda_\theta\setminus S$, 
the sequence can be used in the context of the Preservation Lemma
\ref{preservation}. It can also be used in the context of the Construction Lemma \ref{construction}.
Let us therefore find the Boolean algebras ${\mathfrak A}[A]$ as described
in the statement of the Construction Lemma.
Notice that there are $2^\theta$ many different choices for $A\cap S^\theta_{\neq\omega}$. 

1. Suppose for a contradiction that there is a family $\{C({\rm St} (\alg _\alpha)):\,\alpha <\alpha^\ast\}$
for some $\alpha^\ast<2^\theta$  for some algebras ${\mathfrak A}_\alpha$
of size $\lambda$ which is very positively universal for all $C({\rm St} (\alg ))$ for Boolean algebras $\alg$ of size $\lambda$.
Notice that the assumptions we have made on $\lambda$ imply that $\lambda^{\aleph_0}=\lambda$ so the
size of each $N({\mathfrak A}_\alpha)$ is $\lambda$. Let $\FF$ be the family of all subsets of $\theta$ that appear as invariants of
elements of $\bigcup_{\alpha<\alpha^\ast} N({\mathfrak A}_\alpha)$, hence the size of $\FF$ is $< 2^\theta$, and in particular 
there is $A\subseteq \theta$ a closed set of limit ordinals such that $A\cap S^\theta_{\neq\omega}\notin \FF$. 
Let $\alg=\alg[A]$. Suppose that $T$
is a very positive embedding of $C({\rm St} (\alg ))$ into some $C({\rm St} (\alg _\alpha))$ and let
$\phi$ be an embedding of $N(\alg)$ into $N(\alg_\alpha)$ satisfying that for every $\bar{f}\in {\mathcal A}
\cap C^{N(\alg)}$ we have $\phi(\bar{f})=T(f)$, which exists by Theorem \ref{induced}. Let $E_0$ be a club of $\lambda$ 
as guaranteed by the Preservation Lemma and $E_1$ a club of $\lambda$ as guaranteed by the Construction Lemma, 
let $E=E_0\cap E_1$ and suppose that $\delta$ is such that $D_\delta\subseteq E$. Then by the choice of $E_0$ there is
$\bar{f}$ in $N(\alg)$ whose invariant is $A\cap S^\theta_{\neq\omega}$, but then $\phi(\bar{f})$ also has invariant $A\cap S^\theta_{\neq\omega}$, by the choice of $E_1$
and we have a contradiction with the choice of $A$.

2. Consider the family ${\mathcal G}=\{C({\rm St}(\alg [A])):\, A \mbox{ a closed set of limit ordinals in }\theta\}$. By the argument in 1.
for every $A$, the set $\{B\mbox{ a closed set of limit ordinals in }\theta:\, C({\rm St}(\alg [B]))$
embeds very positively into 
$C({\rm St}(\alg [A]))\}$ has size $<2^\theta$, so clearly every $C({\rm St}(\alg [A]))$ is very positively isomorphic
with $<2^\theta$ many $C({\rm St}(\alg [B]))$. Hence we can choose ${\rm cf}(2^\theta)$ pairwise non very positively isomorphic
elements of ${\mathcal G}$ by a simple induction.
$\eop_{\ref{glavni}}$
\end{proof}
An example of circumstances when Theorem \ref{glavni} applies is when
\[
\theta=\aleph_1, \lambda=\aleph_3, 2^{\aleph_0}=\aleph_1\mbox{ but } 2^{\aleph_1}\ge\aleph_4.
\]

\section{Example}\label{example} We give an example of two 0-dimensional spaces $K$ and $L$ which are not homeomorphic
yet they admit a very-positive isomorphism onto. The example itself was constructed by Grzegorz Plebanek in \cite{Plebanekpos}, Example 5.3
when considering positive onto isomorphisms. 

Let $K$ consist of two disjoint convergent sequences $\langle x_n:\,n<\omega\rangle$ with $\lim_n x_n=x$ and
$\langle y_n:\,n<\omega\rangle$ with $\lim_n y_n=y\neq x$ and let $L$ consist of a single convergent sequence
$\langle z_n:\,n<\omega\rangle$ with $\lim_n z_n=z$. Define $T:C(K)\to C(L)$ by letting for all $n$
\[
Tf(z_0)=f(y),
\]
\[
Tf(z_{2n-1})=\dfrac{f(x_n)+f(y)}{2}, Tf(z_{2n+2})=\dfrac{f(x)+f(y_n)}{2} 
\]
Plebanek shows that $T$ is a positive isomorphism onto $C(L)$ and moreover he calculates the inverse $S=T^{-1}$ which is given by
\[
Sh(y)=h(z_0), Sh(x)=2h(z)-h(z_0),
\]
\[
Sh(x_{n})=2 h(z_{2n-1})-h(z_0), Sh(y_n)=2h(z_{2n})-2h(z)+h(z_0) \mbox{ for } n\ge1.
\]
We shall show that $T$ is a very positive embedding. Considering property (ii) of Definition \ref{strictlypositive}, suppose that
$g\in C(L)\setminus\{0\}$ with $0\le g$, we need to find $h\ge  0$ with $0\le Th\le g$ and $h\neq 0$. Let $t$ in $C(L)$ be
such that $t(z_0)=0$ and for $n\ge 1$ the sequence $\{t(z_n)\}_n$ is
non-negative and not identically 0, converges to $0$ and is bounded by the sequence $\{g(z_n)\}_n$. It is possible to find such a sequence since
$g\ge 0$ and $g\neq 0$. Let $t(z)=0$, hence $t\in C(L)$. Let $h=St$, and we note from the definition of $S$ that $h\ge 0$
and we have $Th=t\le g$.

For the property (iii), we shall have an existential proof of the existence of the $s$ as required. Let ${\mathcal S}$
be the family of all non-negative functions in $C(K)$ for which there is exactly one point with non-zero value, and on that 
point the value is equal to that of $h$. Each element of $\mathcal S$ is clearly definable from $h$. 
We claim that some $s\in {\mathcal S}$ can be chosen to demonstrate (iii). Namely, since we do not
have $Tf \le Th$, we cannot have $f\le h$ by positivity. Hence there is some value $w$ with $h(w)<f(w)$, and by the
continuity of the functions $h$ and $f$, there must be some such $w\in \{x_n, y_n:\,n\in \omega\}$. Then letting $s(w)=h(w)$
and $s(v)=0$ for $v\neq w$ gives a function in $\mathcal S$ and we have $0 \le s\le f$ and $0\neq s$.

\section{Proof of the Construction Lemma}\label{proofCL}
We present a proof of Lemma \ref{construction}. Let $S\subseteq S^\lambda_\theta$ and sequences $\langle c_\delta:\,\delta\in S\rangle$, 
$\langle \PP_\alpha:\,\alpha<\lambda\rangle$ be as in the statement of Theorem \ref{guess square}, while $\langle D_\delta:\,\delta\in S^\lambda_\theta\rangle$ is such that $D_\delta=c_\delta$ for $\delta\in S$. For all the definitions of invariants we use here, the value of the invariant is the same with respect to $\langle c_\delta:\,\delta\in S\rangle$ as it is with respect to 
$\langle D_\delta:\,\delta\in S^\lambda_\theta \rangle$ so we shall not make a difference between the two. We start with a
construction lemma for a certain family of linear orders, as obtained by Kojman and Shelah in \cite{KjSh409}. 
Let us give their definition of the invariants of linear orders:

\begin{definition}\label{invlinear} Suppose that $L$ is a linear order with the universe $\lambda$ and ${\mathfrak L}=
\langle L_\delta:\,\delta<
\lambda\rangle$ is a filtration of $L$. Then for every $\delta\in S$ such that the universe of $L_\delta$ is $\delta$ we define
\[
{\rm inv}_{{\mathfrak L},\delta}(\delta)\deq\{i<\theta:\,(\exists\delta'\in (\alpha^\delta_i, \alpha^\delta_{i+1}])(\forall x\in
 L_{\alpha^\delta_i})\, x\le_L \delta\iff x\le_L\delta'\}.
 \]
\end{definition}

Lemma 3.7 in \cite{KjSh409} proves that under the assumptions we have stated, for every closed set $A$ of limit ordinals in $\theta$ 
there is a linear order $L[A]$ with universe $\lambda$ and a filtration ${\mathfrak L}[A]=
\langle L_\delta[A]:\,\delta<
\lambda\rangle$ of $L[A]$ such that for every $\delta\in S$ with $L_\delta[A]=\delta$ we have 
${\rm inv}_{{\mathfrak L[A]},\delta}(\delta)=A$\footnote{Lemma 3.7 in \cite{KjSh409} also states the assumption $2^\theta>\lambda$, but that
assumption is not used in the proof, it is only needed for the final result in \cite{KjSh409}.}.

The idea of our proof is to transfrom the Kojman-Shelah construction first into a construction of a family of Boolean
algebras of size $\lambda$ and then to use these Boolean algebras to define natural spaces of functions with
appropriate invariants. 

\begin{definition}\label{invboolean} Suppose that ${\mathfrak A}$ is a Boolean algebra with the 
set of generators $\{a_\alpha:\,\alpha<\lambda\}$ and $\bar{{\mathfrak A}}=\langle \alg_\delta:\,\delta<
\lambda\rangle$ is a filtration of $\alg$, while
$\delta\in S$ is such that $\alg_\delta$ is generated by $\{a_\alpha:\,\alpha<\delta\}$. We define 
\[
{\rm inv}_{\bar{{\mathfrak A}},\delta}(a_\delta)\deq\{i<\theta:\,(\exists \delta'\in (\alpha^\delta_i, \alpha^\delta_{i+1}])(\forall
\alpha<\alpha^\delta_i) \]
\[
a_\alpha\cap a_\delta= a_\alpha\cap a'_\delta\mod \alg_{\alpha^\delta_i} \mbox{ and }
a_\alpha^c\cap a_\delta= a_\alpha^c\cap a'_\delta\mod \alg_{\alpha^\delta_i}\},
  \]
 where $a= b \mod {\alg_{\alpha^\delta_i}}$ means that for any element $w$ of $\alg_{\alpha^\delta_i}$ 
 we have $w\le a$ iff $w\le b$.
\end{definition}

\begin{definition}\label{algfromlin} Suppose that $L$ is a linear order with universe $\lambda$. We define
a Boolean algebra $\alg[L]$ as being generated by $\{a_\alpha:\,\alpha<\lambda\}$ freely except for the equations
\begin{equation}\label{translation}
a_\delta\le a_\varepsilon \iff \delta\le_L\varepsilon.
\end{equation}
\end{definition}

Since the equations in (\ref{translation}) are finitely consistent with the axioms of a Boolean algebra  it follows from the compactness theorem that the algebra $\alg[L]$ is well defined. Now we shall see a translation between the calculation of the invariants
of the linear orders and the associated Boolean algebras.

\begin{sublemma}\label{odredadoalgebre} Let $L$ be a linear order on $\lambda$ and $\alg[L]$ the algebra associated
to $L$ as per
Definition \ref{algfromlin}. Let $\bar{L}$ and $\bar{\alg}$ be any filtrations of $L$ and $\alg[L]$ respectively.
Then there is a club $E$ such that for every $\delta\in S\cap E$ we have
\[
{\rm inv}_{\bar{L},\delta}(\delta)={\rm inv}_{\bar{{\mathfrak A}},\delta}(a_\delta)
\]
and moreover, for any $i\in {\rm inv}_{\bar{L},\delta}(\delta)$, this is exemplified by $\delta'$ iff 
$i\in {\rm inv}_{\bar{{\mathfrak A}},\delta}(a_\delta)$ is exemplified by $a_{\delta'}$.
\end{sublemma}

\begin{proof} Let $E$ be the club of $\delta$ such that the universe of $L_\delta$ is $\delta$, $L_\delta$ 
is an elementary submodel of $L$,  $\alg[L]_\delta$ is
generated by $\{a_\alpha:\,\alpha<\delta\}$
and is an elementary submodel of $\alg$.
Suppose that $\delta\in E\cap S$. 

First suppose that $i\in {\rm inv}_{{\mathfrak L},\delta}(\delta)$ as exemplified by $\delta'$. Let $\alpha<
\alpha^\delta_i$, we need to prove $a_\alpha\cap a_\delta= a_\alpha\cap a'_\delta\mod \alg_{\alpha^\delta_i} \mbox{ and }
a_\alpha^c\cap a_\delta= a_\alpha^c\cap a'_\delta\mod \alg_{\alpha^\delta_i}$.

\underline{Case 1.} $\alpha<_L \delta$. Hence by the choice of $\delta'$ we have $\alpha<_L\delta'$
and $a_\alpha\le a_\delta$, $a_\alpha\le a_\delta'$. Therefore $a_\alpha\cap a_\delta=a_\alpha\cap a_\delta'=a_\alpha$.

Suppose that $z>0$ is in $\alg_{\alpha^\delta_i}$ and satisfies $z\le a_\alpha^c \cap a_\delta$. By the Disjunctive Normal
Form for Boolean algebras, we can assume that $z=\bigvee_{i\le n}\bigwedge_{j\le k_i} a^{l(i,j)}_{\beta(i,j)}$ for some 
$l(i,j)\in \{0,1\}$ and $\beta(i,j)<\alpha^\delta_i$. It suffices to prove that for every $i$ we have 
$\bigwedge_{j\le k_i} a^{l(i,j)}_{\beta(i,j)}\le a_\alpha^c \cap a_\delta'$. Fix an $i$ and without loss of generality assume that 
$\bigwedge_{j\le k_i} a^{l(i,j)}_{\beta(i,j)}>0$, as otherwise the conclusion is trivial.

Let $A_l=\{j\le k_i:\,\beta(i,j)=l\}$, for $l\in\{0,1\}$. Let $\beta_1$ be the $L$-minimal element of $A_1$, hence 
$\bigwedge_{j\in A_1} a^{l(i,j)}_{\beta(i,j)}=a_{\beta_1}$. Let $\beta_0$ be the $L$-maximal element of $A_0$, hence 
$\bigwedge_{j\in A_0} (a^{l(i,j)}_{\beta(i,j)})^c=[ \bigvee_{j\in A_0} a^{l(i,j)}_{\beta(i,j)}]^c=a_{\beta_0}^c$. In conclusion,
$\bigwedge_{j\le k_i} a^{l(i,j)}_{\beta(i,j)}=a_{\beta_0}^c \cap a_{\beta_1}$. Since we have assumed that 
$\bigwedge_{j\le k_i} a^{l(i,j)}_{\beta(i,j)}>0$, we cannot have $a_{\beta_1}\le a_{\beta_0}$, equivalently $\beta_1\le_L \beta_0$.
Hence we have $\beta_0 <_L  \beta_1$. Similarly, since $a_{\beta_1}\cap a_\alpha^c>0$ we can conclude that $\alpha<_L\beta_1$.
Finally, if we had $\beta_0>_L \delta$, then we would obtain $a_{\beta_0}^c\le a_\delta^c$, in contradiction with 
$0< a_{\beta_0^c}\cap a_{\beta_1}\le a_\delta$, and therefore $\beta_0<_L \delta$.

Suppose now that $\delta<_L \beta_1$.
Therefore $a_{\beta_1}\cap a_\delta^c\neq 0$. On the other hand, $a_{\beta_0}^c\cap a_{\beta_1}\cap a_\delta^c\le
a_\delta\cap a_{\alpha^c}\cap a_{\delta^c}=0$ and hence we must have $a_{\beta_0}\cap a_{\beta_1}\cap a_\delta^c>0$,
which, taking into account $\beta_0<_L\beta_1$ gives that $a_{\beta_0}\cap a_\delta^c>0$ and hence $\delta<_L \beta_0$,
a contradiction. Hence we have $\beta_1<_L\delta$. 
By the choice of $\delta'$ we have $\beta_1<_L\delta'$ and hence $a_{\beta_1}\le a_{\delta'}$ and in particular 
$a_{\beta_0}^c \cap a_{\beta_1}\le a_{\delta'}$, as required.
Since the roles of $\delta$ and $\delta'$ in this proof
were symmetric we can prove in the same way that for any $z>0$ is in $\alg_{\alpha^\delta_i}$ which satisfies $z\le a_\alpha^c \cap a_\delta'$ we also have $z\le a_\alpha^c \cap a_\delta$.

\underline{Case 2.} $\alpha>_L \delta$, so $\alpha>_L \delta'$ by the choice of $\delta'$. We have $a_{\delta}^c\ge a_{\alpha}^c$,
so $a_{\alpha}^c\cap a_\delta=0$ and similarly $a_{\alpha}^c\cap a_\delta'=0$. We also have $a_\alpha\cap a_{\delta}=a_\delta$
and similarly for $\delta'$, hence we need to prove that $a_\delta = a_{\delta'}\mod \alg_{\alpha^\delta_i}$. As in Case 1, it suffices
to show that for every $\beta_0, \beta_1<\alpha^\delta_i$ with $0< a_{\beta_0}^c\cap a_{\beta_1}<a_\delta$, we have 
$a_{\beta_0}^c\cap a_{\beta_1}\le a_\delta'$ (the equality cannot occur), and vice versa. Let us start with the forward direction.  As before, from $0< a_{\beta_0}^c\cap a_{\beta_1}$ we conclude 
$\beta_0<_L \beta_1$. Also, if $\beta_0>_L\delta$ then we have $a_{\beta_0}^c\le a_\delta^c$, contradicting that $a_{\beta_0}^c
\cap a_\delta>0$. Hence $\beta_0<_L\delta$.

If $\beta_1<_L\delta$ then $\beta_1<_L\delta'$ so $a_{\beta_1}\le a_\delta'$ and hence 
$a_{\beta_0}^c\cap a_{\beta_1}\le a_\delta'$, as required. So assume that $\delta<_L \beta_1$. Hence $a_{\beta_1}> a_\delta$
and so $a_{\beta_0}^c\cap a_{\beta_1}> a_{\beta_0}^c\cap a_\delta\ge a_{\beta_0}^c\cap a_{\beta_1}$, a contradiction.
This finishes the proof of the forward direction, and the other direction follows from the symmetry of the roles of $\delta$ and $\delta'$ in the proof.

Now suppose that $i\in {\rm inv}_{{\mathfrak A},\delta}(a_\delta)$ as exemplified by $a_\delta'$. Let $\alpha<
\alpha^\delta_i$,
we need to prove $\alpha<_L\delta\iff \alpha<_L\delta'$. If $\alpha<_L \delta$ then $a_\alpha< a_\delta$ hence
$a_\alpha< a_{\delta'}$ by the assumption, and hence $\alpha<_L \delta'$ by the definition of $\alg[L]$. The other direction follows
by symmetry.
$\eop_{\ref{odredadoalgebre}}$
\end{proof}

\begin{sublemma}\label{fromalgtonat} Let $\alg[L]$ be one of the
algebras described in the above and $\bar{\alg}$ its filtration. 
Then there is a club $E$ of $\lambda$
such that for every $\delta\in S$ with $D_\delta\subseteq E$, we have that 
\[
{\rm inv}_{\bar{\alg},\delta}(\chi_{[a_\delta]})={\rm inv}_{\bar{{\mathfrak A}},\delta}(a_\delta)
\]
and moreover, for any $i\in {\rm inv}_{\bar{{\mathfrak A}},\delta}(\chi_[a_\delta])$, this is exemplified by $\chi_{[a_\delta']}$ iff 
$i\in {\rm inv}_{\bar{{\mathfrak A}},\delta}(a_\delta)$ is exemplified by $a_{\delta'}$. Here, the invariant on the left refers to the
invariant in the natural space $N(\alg)$ and the invariant on the right to the invariant in the algebra $\alg$. The notation
$\chi_{[a]}$ is used for the sequence $\langle  \chi_{[a]}, \chi_{[a]}, \chi_{[a]}, \ldots\rangle$ in $C^{N(\alg)}$.
\end{sublemma}

\begin{proof} Let ${\mathfrak M}^\ast$ be a model consisting of $L$, $\mathfrak A$, two disjoint copies of the $\omega$-sequences of the simple functions
on $\lambda$ with rational coefficients, interpreted as
the elements of $N(\mathfrak A)$ and all the symbols of  $N(\mathfrak A)$ with induced interpretations induced from these models. 
Recall the assumption that $\forall\kappa<\lambda$ we have $\kappa^{\aleph_0}<\lambda$ and notice that it 
implies that there is a club $E_0$ of $\lambda$ such that for every $\delta\in E_0$ of cofinality $>\aleph_0$,
the model ${\mathfrak M}^\ast\rest\delta$ is $\aleph_1$-saturated in ${\mathfrak M}^\ast$ , that is it realizes all the types
with countably many parameters in ${\mathfrak M}^\ast\rest\delta$ which are realized in ${\mathfrak M}^\ast$.
Let $\bar{L}$ be any filtration of $L$, let $E\subseteq E_0$ be a club witnessing  Sublemma
\ref{odredadoalgebre}, and let $\delta\in S$ be such that $D_\delta\subseteq E$. 

Suppose $i\in {\rm inv}_{\bar{{\mathfrak A}},\delta}(a_\delta)$ as exemplified by $a_{\delta'}$ but $\bar{g}\ge 0$
with $\bar{g}\in C^{N(\alg_{\alpha^\delta_i})}\setminus C_0$ and the limit $g$  of $\bar{g}$
satisfies $g\le |\chi_[a_\delta]-\chi_[a_\delta']|=
\chi_{[a_{\delta}\Delta a_{\delta'}]}$. By topping up if necessary (see Definition \ref{topping}) we may assume that each $g_n\ge 0$,
and by throwing away unnecessary elements of $\bar{g}$ we may assume
that every $g_n\neq 0$. 
We can then assume that for each $n$ there are pairwise disjoint $\{b_0^n,\ldots b^{k_n}_n\}\in \alg_{\alpha^\delta_i}$ and
$q_i^n ( i\le k_n)\in {\mathbb Q}^+$ such that $g_n=\Sigma_{i\le k_n} q_i^n \chi_{[b^n_i]}$. Since $||g_n-
\chi_{[a_{\delta}\Delta a_{\delta'}]} ||\into 0$, there has to be a $[b^n_i]$ with a non-empty intersection with 
$[a_{\delta}\Delta a_{\delta'}]$. By applying the Disjunctive Normal form, we can assume that 
$b^n_i=\bigvee_{j\le m}\bigwedge_{o\le o_m} a^{l(m,o)}_{\beta(m,o)}$ for some 
$l(m,o)\in \{0,1\}$ and $\beta(m,o)<\alpha^\delta_i$. Therefore there is $j\le m$ such that 
$\bigwedge_{o\le o_m} [a^{l(m,o)}_{\beta(m,o)}]\cap [a_\delta\Delta a_{\delta'}]\neq \emptyset$. Then we have that 
$\bigwedge_{o\le o_m} [a^{l(m,o)}_{\beta(m,o)}]\cap [a_\delta] \neq \bigwedge_{o\le o_m} [a^{l(m,o)}_{\beta(m,o)}]\cap [a_\delta']$,
and hence there has to be $o\le p_m$ such that $[a^{l(m,o)}_{\beta{(m,o)}}]\cap [a_\delta] \neq
[a^{l(m,o)}_{\beta{(m,o)}}]\cap [a_\delta']$. It follows that $[a^{1-l(m,o)}_{\beta{(m,o)}}]\cap [a_\delta] \neq
[a^{1-l(m,o)}_{\beta{(m,o)}}]\cap [a_\delta']$. Let $\beta=\beta(m,0)$. 
From the choice of $E$, using Sublemma \ref{odredadoalgebre}
we have that for $R\in \{<_L, >_L\}$, $\beta R \delta$ iff $\beta R \delta'$.
We go through a case analysis like in the proof of Sublemma \ref{odredadoalgebre}. 
If $\beta<_L\delta$ then we have $\beta<_L\delta'$ so $[a_\beta]\cap [a_{\delta}]=[a_\beta]=[a_\beta]\cap [a_{\delta'}]$,
a contradiction. If $\beta>_L \delta$ then $[a_\beta^c]\cap [a_{\delta}]=\emptyset =[a_\beta^c]\cap [a_{\delta'}]$,
a contradiction. Therefore $i\in {\rm inv}_{\bar{\alg},\delta}(\chi_{[a_\delta]})$.

\begin{Claim}\label{special witness}
Suppose that $i\in {\rm inv}_{\bar{\alg},\delta}(\chi_{[a_\delta]})$ as exemplified by some $\bar{f}$. Without loss of generality we can assume that $\bar{f}=\chi_{[a_\delta']}$ for some $\delta'$.
\end{Claim}

\begin{Proof of the Claim} First let us notice that if $f=\lim_n f_n$ then for $f'=\min\{f,1\}$ we have $|\chi_{[a_\delta]}-f'|
\le |\chi_{[a_\delta]}-f|$ so we can without loss of generality assume that $f\le 1$. Similarly we can assume that $f\ge 0$,
and then by applying a similar logic we can also assume that $0\le f_n\le 1$ for all $n$ and that $f_n\neq 0$.
Each $f_n$ is a simple function with rational coefficients defined on (without loss of generality) disjoint
basic clopen sets of the form $[a_\beta^{l_\beta}]$
where $\beta<\alpha^{\delta}_{i+1}$ and $l_\beta<2$. Let $\{\beta_n:\,n<\omega\}$ enumerate all the relevant
$\beta$. For each $n$ and $R\in \{<_L,>_L\}$ let $j^n_R$ be the truth value of $``a_{\beta_n}R a_\delta$".
Consider the following sentence
with parameters $\bar{f}$, $\alpha^\delta_i$ and the elements of $\{\beta_n:\,n<\omega\}$: there is $\beta$ such that
\begin{itemize}
\item  for all $\bar{g}\in N(\alg_{\alpha^\delta_i})$
if $0\le \lim_n g_n\le |\lim f_n- \chi_{[a_\beta]}|$, we have $\bar{g}\in C^0$,
\item for all $n$ and $R\in \{<_L,>_L\}$ we have $a_{\beta_n} R a_\beta$ iff $j^n_R=1$.
\end{itemize}
This sentence is true as exemplified by 
$\delta$, so by the choice of $E_0$ it is true in ${\mathfrak M}^\ast\rest{\alpha^\delta_{i+1}}$, say as exemplified by $\delta'$.
Let us note that in $L$ we have
$\delta<_L \delta'$ or $\delta'<_L\delta$, let us assume that $\delta'<_L\delta$, 
as the other case is symmetric.
We claim that $\chi_{[a_\delta']}$ exemplifies that $i\in {\rm inv}_{\bar{\alg},\delta}(\chi_{[a_\delta]})$. If not, we can find
$\bar{g}\in C^{N(\alg_{\alpha^\delta_i})}\setminus C^0$ with $0\le \bar{g}$ and $g\deq\lim_n g_n\le |\chi_{[a_\delta]}-
\chi_{[a_\delta']}|=\chi_{[a_\delta\Delta a_\delta']}\le 1$. By the triangle inequality it follows that 
\[
g\le |\chi_{[a_\delta]}-f|+ |f-\chi_{[a_\delta']}|.
\]
We have that for every $x$ both $|\chi_{[a_\delta]}-f|(x)$ and $|\chi_{[a_\delta]}-f|(x)$ are equal to $f(x)$ if $x\in [a_\delta^c]$ and
$1-f(x)$ if $x\in [a_{\delta'}]$. The possible difference is on $[a_\delta\setminus a_{\delta'}]$, where the former function is
equal to $1-f(x)$ and the latter to $f(x)$. We now claim that $f$ is constant on $[a_\delta\setminus a_{\delta'}]$.

Clearly, it suffices to show that each $f_n$ is constant on $[a_\delta\setminus a_{\delta'}]$, and by the choice of 
$\{\beta_n:\,n<\omega\}$, it suffices to show that for each $l<2$ and each $n$, $\chi_{[a^l_{\beta_n}]}$ is 
constant on $[a_\delta\setminus a_{\delta'}]$. Let $\beta=\beta_n$ for some $n$. If $\beta \le _L \delta'$ then
$a_\beta\le a_{\delta'}$ so $\chi_{[a_{\beta}]}$ is constantly 0 on $[a_\delta\setminus a_{\delta'}]$. In addition,
we have $a_\beta^c\ge a_{\delta'}^c\ge (a_\delta\setminus a_\delta')$ so $\chi_{[a_{\beta}^c]}$ is constantly 1 on $[a_\delta\setminus a_{\delta'}]$. If $\beta\ge_L\delta'$ then $\beta\ge_L\delta$ by the choice of $\delta'$ so $a_\beta\ge a_\delta$ and
hence $\chi_{[a_{\beta}]}$ is constantly 1 on $[a\delta\setminus a_{\delta'}]$. In addition,
$a_\beta^c\le a_{\delta}^c$ so $\chi_{[a_{\beta}^c]}$ is constantly 0 on $[a_\delta\setminus a_{\delta'}]$, and the statement is proved.

Let $\varepsilon\in [0,1]$ be such that $f$ is constantly $\varepsilon$ on $[a_\delta\setminus a_{\delta'}]$. Say $\varepsilon\le 1/2$
as the other case is symmetric. Hence $\max\{f, 1-f\}$ on $a_\delta\setminus a_{\delta'}$ is $1-f$ and in particular
we have 
\[
0\le g\le |\chi_{[a_\delta]}-f|+ |f-\chi_{[a_\delta']}|\le 2\cdot |\chi_{[a_\delta]}-f|
\]
and hence $1/2\cdot g$ is a function which contradicts that $\bar{f}$ exemplifies that $i\in {\rm inv}_{\bar{\alg},\delta}(\chi_{[a_\delta]})$. A contradiction and hence $\chi_{[a_\delta']}$ exemplifies that $i\in {\rm inv}_{\bar{\alg},\delta}(\chi_{[a_\delta]})$ as
required.
$\eop_{\ref{special witness}}$
\end{Proof of the Claim}

Now suppose that $i\in {\rm inv}_{\bar{\alg},\delta}(\chi_{[a_\delta]})$ and assume without loss of generality by Claim 
\ref{special witness}
that $\bar{f}=\chi_{[a_\delta']}$ for some $\delta'$. We need to prove that $i\in {\rm inv}_{\bar{{\mathfrak A}},\delta}(a_\delta)$
as exemplified by $a_{\delta'}$. But indeed, if for some $\beta<\alpha^\delta_i$ and $l<2$  we have that $a_\beta^l\cap a_\delta
\neq a_\beta\cap a_{\delta'}\mod \alg^{\alpha^\delta_i}$, then there is $w>0$ in  $\alg^{\alpha^\delta_i}$ with
$w\le a_\delta\Delta a_{\delta'}$ and then clearly the sequence $\langle \chi_{[w]},  \chi_{[w]}, \ldots \rangle$ is not in $C^0$
and is below $\langle \chi_{[a_\delta\Delta a_{\delta'}]},  \chi_{[a_\delta\Delta a_{\delta'}]}, \ldots \rangle$, a contradiction.
$\eop_{\ref{fromalgtonat}}$
\end{proof}

\bibliographystyle{plain}
\bibliography{biblio}

\end{document}